\documentclass{amsproc}

\usepackage{amsmath,amsfonts,amssymb,amsthm}
\usepackage{cite}
\usepackage{graphicx}

\newtheorem{theorem}{Theorem}[section]

\newtheorem{lemma}[theorem]{Lemma}
\newtheorem{proposition}[theorem]{Proposition}
\theoremstyle{definition}

\theoremstyle{remark}
\newtheorem{remark}[theorem]{Remark}

\numberwithin{equation}{section}

\newcommand{\A}[1]{\overset{#1}A}
\newcommand{\B}[1]{\overset{#1}B}
\newcommand{\AB}[1]{\overset{#1}{\overline{[A,B]}}}
\newcommand{\df}{\frac{\delta f}{\delta\xi}}

\newcommand\pd[2]{\frac{\partial #1}{\partial #2}}
\newcommand\od[2]{\frac{d #1}{d #2}}
\newcommand{\abs}[1]{\lvert#1\rvert}
\newcommand{\norm}[1]{\left\| #1 \right\|}

\DeclareMathOperator\im{Im}
\DeclareMathOperator\re{Re}

\DeclareMathOperator\const{const}

\begin{document}

\title[Functions of noncommuting operators in an asymptotic
problem]{Functions of Noncommuting Operators\\
in an Asymptotic Problem\\ for a 2D Wave Equation with Variable
Velocity\\ and Localized Right-Hand Side}

%----------Author 1
\author[S.~Dobrokhotov]{Sergei~Dobrokhotov}
\address{A.~Ishlinsky Institute for Problems in Mechanics,
Russian Academy of Sciences, Moscow\\
Moscow Institute of Physics and Technology}
\email{dobr@ipmnet.ru}
\thanks{Supported by RFBR grants nos.~11-01-00973-a and~11-01-12058-ofi-m-2011
and by the scientific agreement between the Department of Physics,
University of Rome ``La Sapienza," Rome, and A.~Ishlinsky Institute
for Problems in Mechanics, Russian Academy of Sciences, Moscow}

%----------Author 2
\author[D.~Minenkov]{Dmitrii~Minenkov}
\address{A.~Ishlinsky Institute for Problems in Mechanics,
Russian Academy of Sciences, Moscow\\
Moscow Institute of Physics and Technology}
\email{minenkov\_ds@list.ru}

%----------Author 3
\author[V.~Nazaikinskii]{Vladimir~Nazaikinskii}
\address{A.~Ishlinsky Institute for Problems in Mechanics,
Russian Academy of Sciences, Moscow\\
Moscow Institute of Physics and Technology}
\email{nazay@ipmnet.ru}

%----------Author 4
\author[B.~Tirozzi]{Brunello Tirozzi}
\address{Department of Physics, University ``La Sapienza,'' Rome}
\email{brunello.tirozzi@roma1.infn.it}

\renewcommand{\subjclassname}{Mathematics Subject Classification
(2010)}

\subjclass[2010]{Primary 35L05; Secondary 81Q20, 35Q35}

\keywords{Wave equation, localized solutions, asymptotics,
noncommutative analysis, canonical operator}

\date{February 6, 2012}

\dedicatory{Dedicated to Vladimir Rabinovich}

\begin{abstract}
In the present paper, we use the theory of functions of
noncommuting operators, also known as noncommutative analysis
(which can be viewed as a far-reaching generalization of
pseudodifferential operator calculus), to solve an asymptotic
problem for a partial differential equation and show how, starting
from general constructions and operator formulas that seem to be
rather abstract from the viewpoint of differential equations, one
can end up with very specific, easy-to-evaluate expressions for the
solution, useful, e.g., in the tsunami wave problem.
\end{abstract}

\maketitle

\tableofcontents

\section{Introduction}

In the present paper, we use the theory of functions of
noncommuting operators \cite{14,15,18}, aka noncommutative analysis
(which can be viewed as a far-reaching generalization of
pseudodifferential operator calculus), to solve an asymptotic
problem for a partial differential equation and show how, starting
from general constructions and operator formulas that seem to be
rather abstract from the viewpoint of differential equations, one
can end up with very specific, easy-to-evaluate expressions for the
solution, useful, e.g., in the tsunami wave problem.

We consider the Cauchy problem with zero initial data for a 2D wave
equation with variable velocity and with right-hand side localized
near the origin in space and decaying in time. One physical
interpretation of this problem is that it describes, in the linear
approximation, the propagation of tsunami waves generated by local
vertical displacements of the ocean bottom (see
\cite{1,2,3,9,10,17,25,26,27} and also \cite{7,8,24,5,6,12} and the
bibliography therein). Normally, the diameter of the region where
these displacements occur (some tens to a hundred of kilometers) is
much smaller than the distance traveled by the waves (thousands of
kilometers), and their ratio, $\mu$, can serve as a small
parameter. Accordingly, we are interested in the asymptotics of the
solution as $\mu\to0$. In the simplest \textit{piston model} of
tsunami generation, the bottom displacement occurs instantaneously
at $t=0$. This corresponds to a right-hand side of the form
$\delta'(t)v(x)$, where $\delta(t)$ is the Dirac delta function, and the
problem is immediately equivalent, via Duhamel's principle, to the
Cauchy problem for the homogeneous wave equation with initial data
$v(x)$ for the unknown function itself and zero initial data for
its $t$-derivative. Fairly explicit asymptotic solution formulas
suitable for easy implementation in \textsl{Wolfram Mathematica}
\cite{20} were constructed and analyzed for the latter problem in
\cite{5,6,24,7,8,12} on the basis of a generalization of Maslov's
canonical operator \cite{13,14}. Now assume we wish to take into
account the fact that the ocean bottom displacement evolves in time
rather than happens instantaneously. Then it is natural to consider
a right-hand side of the form $g'(t)v(x)$, where $g(t)$ is some
smooth approximation to the delta function. An analysis shows that
the solution can be represented as the sum of two parts, a
propagating part, which travels along the characteristics, and a
transient part, which is localized in the vicinity of the origin
and decays in time. The propagating part can further be represented
as the solution of the Cauchy problem for the homogeneous wave
equation with initial data obtained from $v(x)$ by application of
certain functions $f(L)$ of the spatial part $L$ of the wave
operator, where the corresponding symbols $f(\xi)$ are given by
simple formulas expressing them via the Fourier transform of
$g(t)$. These initial data, also localized near the origin, will be
referred to as the \textit{equivalent source functions}. The
transient part of the solution is given by a formula similar to
those for the equivalent source functions with the only difference
that the function $f(\xi)$ additionally depends on time as a
parameter. The transient part is apparently not so important in
tsunami wave analysis, but nevertheless it might be useful from the
viewpoint of satellite registration of tsunami waves
\cite{25,26,27}. Since, as was mentioned above, the asymptotic
formulas for the solution of the Cauchy problem with localized
initial data for the homogeneous wave equation are already known
from \cite{5,6,24,7,8,12}, we see that the only remaining thing is
to compute $f(L)v$ for all these functions $f(\xi)$. It is here
that noncommutative analysis comes fully into play. Note that $L$
is an operator with variable coefficients, and so computing the
function $f(L)$ efficiently may prove quite a challenging task.
However, all we actually need is the asymptotics of $f(L)v$, and
methods of noncommutative analysis permit one to prove that
$f(L)v=f(L_0)v$ plus an asymptotically small remainder, where $L_0$
is obtained from $L$ by freezing the coefficients at the origin.
Now computing $f(L_0)v$ is a breeze, because $f(L_0)$ is conjugate
by the Fourier transform to the operator of multiplication by the
function $f(\sigma_{L_0}(p))$, where $\sigma_{L_0}(p)$ is the symbol
of~$L_0$.

The one-dimensional counterpart of the problem studied in the
present paper was considered in~\cite{4}. In the two-dimensional
case, the results were announced in~\cite{16}, where the proofs
were partly only sketched and partly absent altogether. Here we
develop and refine these results and give complete proofs. Finally,
note that we deal with the setting in which the wave propagation
velocity is assumed to vanish nowhere. The case it which it
vanishes (as it happens on the coastline in the tsunami run-up
problem) is much more complicated. The asymptotics of solutions of
such degenerate problems in some special cases was considered in
the spirit of the approach of \cite{5,6,24,7,8,12} in
\cite{21,22,23} (see also references therein); in the present
paper, we restrict ourselves to wave propagation in open ocean.

The outline of the paper is as follows. In Sec.~\ref{s1}, we give a
detailed statement of the mathematical problem and write out
well-known formulas expressing the solution in operator form. Using
these formulas, we split the solution into the sum of the
propagating and transient parts. Section~\ref{s2} presents simple
formulas for the asymptotics of the solution. The proofs of the
theorems stated in this section depend on the results presented in
Sec.~\ref{s4}, which is the most important part of the paper and
where the asymptotics of the equivalent source functions and the
transient part of the solution are computed with the use of the
noncommutative analysis machinery. Finally, Sec.~\ref{s5} provides
two simple examples; all computations and visualizations in these
examples have been done with \textsl{Wolfram Mathematica}.

\section{Exact solution}\label{s1}

\subsection{Statement of the problem}\label{ss11}

Consider the Cauchy problem for the wave equation
\begin{equation}\label{eq1.1}
 \frac{\partial^2\eta}{\partial t^2}-\frac\partial{\partial
x_1}\biggl(c^2(x)\frac{\partial \eta}{\partial
x_1}\biggr)-\frac\partial{\partial
x_2}\biggl(c^2(x)\frac{\partial\eta}{\partial x_2}\biggr)
  =Q, \qquad t\ge0,
\end{equation}
with the initial conditions
\begin{equation}\label{eq1.2}
   \eta|_{t=0}=0, \qquad \eta_t|_{t=0}=0,
\end{equation}
where $x=(x_1,x_2)\in\mathbf{R}^2$, $\eta=\eta(x,t)$ is the unknown
function, $c(x)$ is an everywhere positive smooth function
stabilizing at infinity,\footnote{That is, $c(x)=\const>0$ for
sufficiently large $\abs{x}$.} and the right-hand side $Q=Q(x,t)$
depends on two parameters $\lambda,\mu>0$ and has the form
\begin{equation}\label{eq1.1b}
    Q(x,t)=\lambda^2g_0'(\lambda t)V\biggl(\frac x\mu\biggr)
\end{equation}
with some smooth real functions $V(y)$, $y\in\mathbf{R}^2$, and
$g_0(\tau)$, $\tau\in[0,\infty)$, such that
\begin{gather}\label{eq1.1c}
    \abs{V^{(\alpha)}(y)}\le C_\alpha(1+\abs{y})^{-\abs{\alpha}-\varkappa},\qquad
    \abs{\alpha}=0,1,2,\dots\,,\\
\label{eq1.1d}
        g_0(0)=0,\qquad
    \int_0^\infty g_0(\tau)\,d\tau=1,\qquad
    \abs{g_0^{(k)}(\tau)}\le C_ke^{-\nu\tau},\quad k=0,1,2,\dots\,,
\end{gather}
for some $\varkappa>1$, $\nu>0$, and positive constants $C_\alpha$ and
$C_k$.
\begin{remark}\label{rem-0}
One can also consider the case in which $g_0(\tau)$ decays as some
(sufficiently large) negative power of $\tau$ as $\tau\to\infty$.
In this case, the estimates are somewhat more awkward, and we
restrict ourselves to the case of the physically natural
exponential decay~\eqref{eq1.1d} for the sake of clarity.
\end{remark}

Our aim is to find the asymptotics as $\mu\to0$ of the solution of
problem \eqref{eq1.1} on an arbitrary finite time interval
uniformly with respect to $\lambda$ in the region
\begin{equation}\label{eq1.1e}
    \lambda\mu>\const>0.
\end{equation}
This will be done in Secs.~\ref{s2} and~\ref{s4}, and in the
present section we deal with the exact solution of the problem.

\subsection{Physical interpretation and examples of right-hand
sides}\label{ss12}

First, speaking in terms of the physical interpretation given in
the introduction, let us explain the meaning of the parameters
$\lambda$ and $\mu$ and condition~\eqref{eq1.1e}. The right-hand side
$Q(x,t)$ describes the time evolution (the factor $\lambda^2g_0'(\lambda
t)$) and the spatial shape (the factor $V(x/\mu)$) of the
perturbation (the tsunami source). In view of~\eqref{eq1.1d}, $\lambda$
characterizes the decay rate of the perturbation, so that
$1/\lambda\sim t_0$, where $t_0$ is the mean lifetime of the
perturbation. The small parameter $\mu$ characterizes the source
size $r_0$, $\mu\sim r_0$. We see that the product $\lambda\mu=r_0/t_0$
has the dimension of velocity and rewrite condition~\eqref{eq1.1e}
in the form
\begin{equation}\label{eq1.1ea}
    \frac{c_0}{\lambda\mu}\equiv\frac{c_0t_0}{r_0}\le\omega_0,
\end{equation}
where $c_0=c(0)$, the wave propagation velocity at the origin, is
taken to represent the typical wave propagation velocity in the
problem and $\omega_0$ is some dimensionless constant. This has a very
clear meaning: the waves excited by the perturbation cannot travel
too far before the perturbation dies out; they only cover a
distance $(c_0t_0)$ of the same order of magnitude as the diameter
$r_0$ of the perturbation region. We introduce the ratio
\begin{equation}\label{omega}
    \omega=\frac{c_0}{\lambda\mu},
\end{equation}
so that condition~\eqref{eq1.1ea} (and hence \eqref{eq1.1e})
becomes
\begin{equation}\label{omega1}
    \omega<\omega_0.
\end{equation}
Mathematically, condition~\eqref{omega1} means that the parameter
$\lambda$ is large (at least of the order of $\mu^{-1}$) as $\mu\to0$.
Note that, in view of the first two conditions in~\eqref{eq1.1d},
$\lambda g_0 (\lambda t)\to \delta(t)$ and $\lambda^2g_0'(\lambda t)\to\delta'(t)$ as
$\lambda\to\infty$.

In what follows, the dependence on the parameters $\lambda$ and $\mu$
is sometimes not immediately important to the argument, and in such
cases we often ``hide'' these parameters by using the notation
\begin{equation}\label{notat}
    g(\tau)=\lambda g_0(\lambda\tau),\quad v(x)=V\biggl(\frac
    x\mu\biggr),\quad\text{so that}\quad
    Q(x,t)=g'(t)v(x).
\end{equation}

Next, let us give specific examples of right-hand sides $Q(x,t)$.
In practice, the actual ocean bottom displacement is known neither
in much detail nor very precisely, because the corresponding
measurements are impractical or impossible (cf.~\cite{25,26,27}).
This results in certain freedom, which can be turned into an
advantage. Namely, when constructing the function
$Q(x,t)=g'(t)v(x)$ to be used in the analytical-numerical
simulation according to the model~\eqref{eq1.1}, one should take
ansatzes that, on the one hand, fit the general information
available about the source shape and evolution and, on the other
hand, can be handled efficiently in the computations. (The latter
includes the requirement that these functions, as well as their
Fourier transforms, be given by closed-form expressions, which
permits one to reduce the amount of numerical computations in favor
of the less time-consuming analytical transformations.)

A useful class of functions $V(y)$ satisfying~\eqref{eq1.1c} is
described by the expression~\cite{12,24,11}, generalizing
\cite{9,10,2},
\begin{equation}\label{eq1.1f}
    V(y)=A\biggl(1+\biggl(\frac{y_1}{b_1}\biggr)^2+\biggl(\frac{y_2}{b_2}\biggr)^2\biggr)^{-3/2},
\end{equation}
where $A$, $b_1$, and $b_2$ are real parameters. The Fourier
transform of this function is remarkably simple,
\begin{equation}\label{FtV}
    \widetilde V(p)=Ab_1b_2e^{-\sqrt{b_1^2p_1^2+b_2^2p_2^2}},
\end{equation}
One can further apply a differential operator
\begin{equation*}
 \widehat{P}=P\biggl(\frac{\partial}{\partial y_1},\frac{\partial}
            {\partial y_2}\biggr)
\end{equation*}
with constant coefficients to the function $V$ and then rotate the
coordinate system by some angle $\theta$, thus obtaining a broad
variety of functions of the form
\begin{equation*}\label{1.4}
V_{P, \theta}(y)= [\widehat P V](T(\theta)y), \qquad T(\theta) \equiv
\begin{pmatrix} \cos\theta & \sin\theta
\\-\sin\theta & \cos\theta \end{pmatrix},
\end{equation*}
satisfying~\eqref{eq1.1c}. Such functions model elliptic-shaped
sources of various eccentricity and various direction of axes with
a wavy relief depending on the differential operator $\widehat P$ (see
\cite{5,6}). Figure~\ref{Fig_V_tildeV} shows the graph of $V(y)$
rotated by an angle of $\pi/10$ and of its Fourier transform.

\begin{figure}
\centering
\includegraphics[width=.48\textwidth]{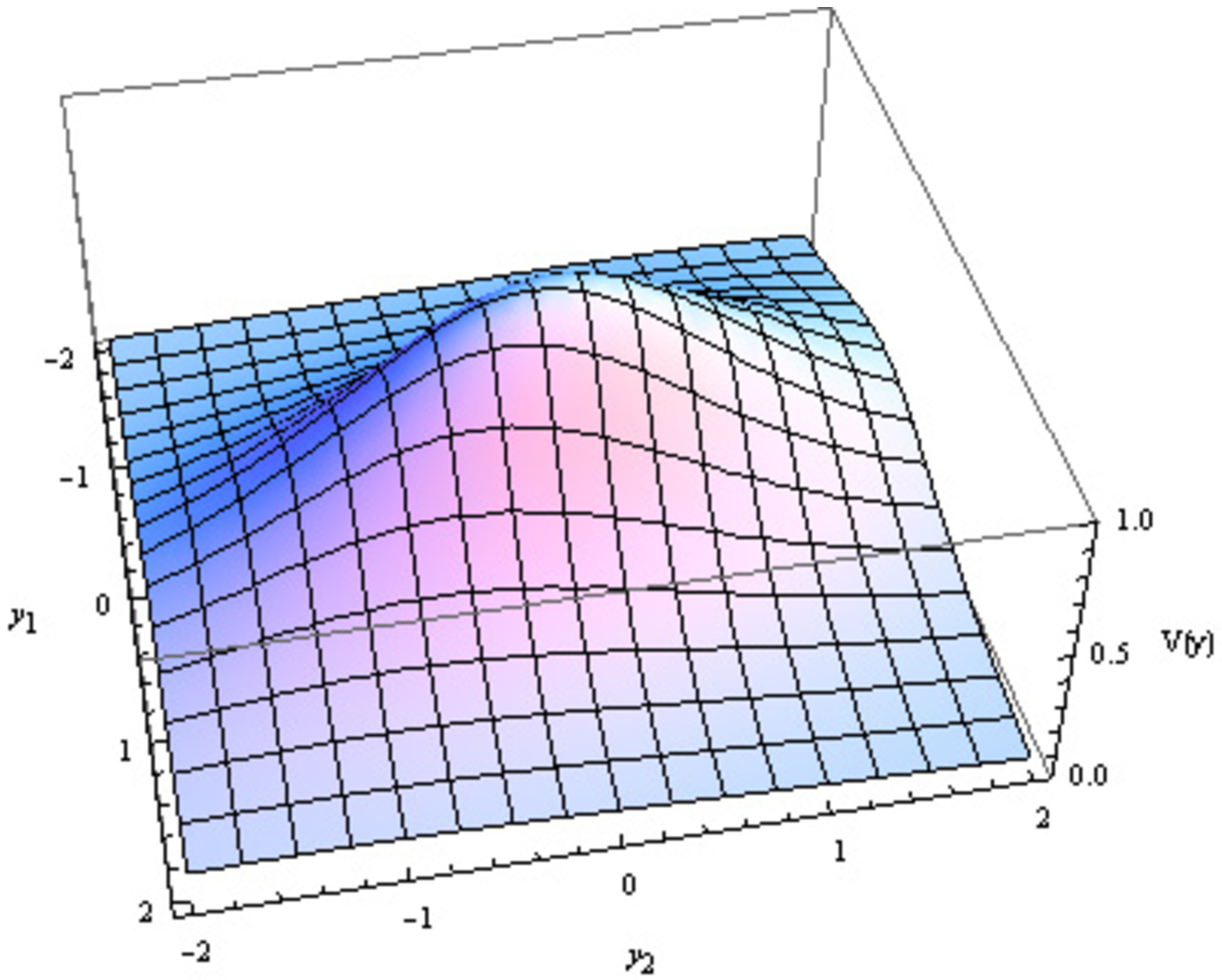}
\includegraphics[width=.48\textwidth]{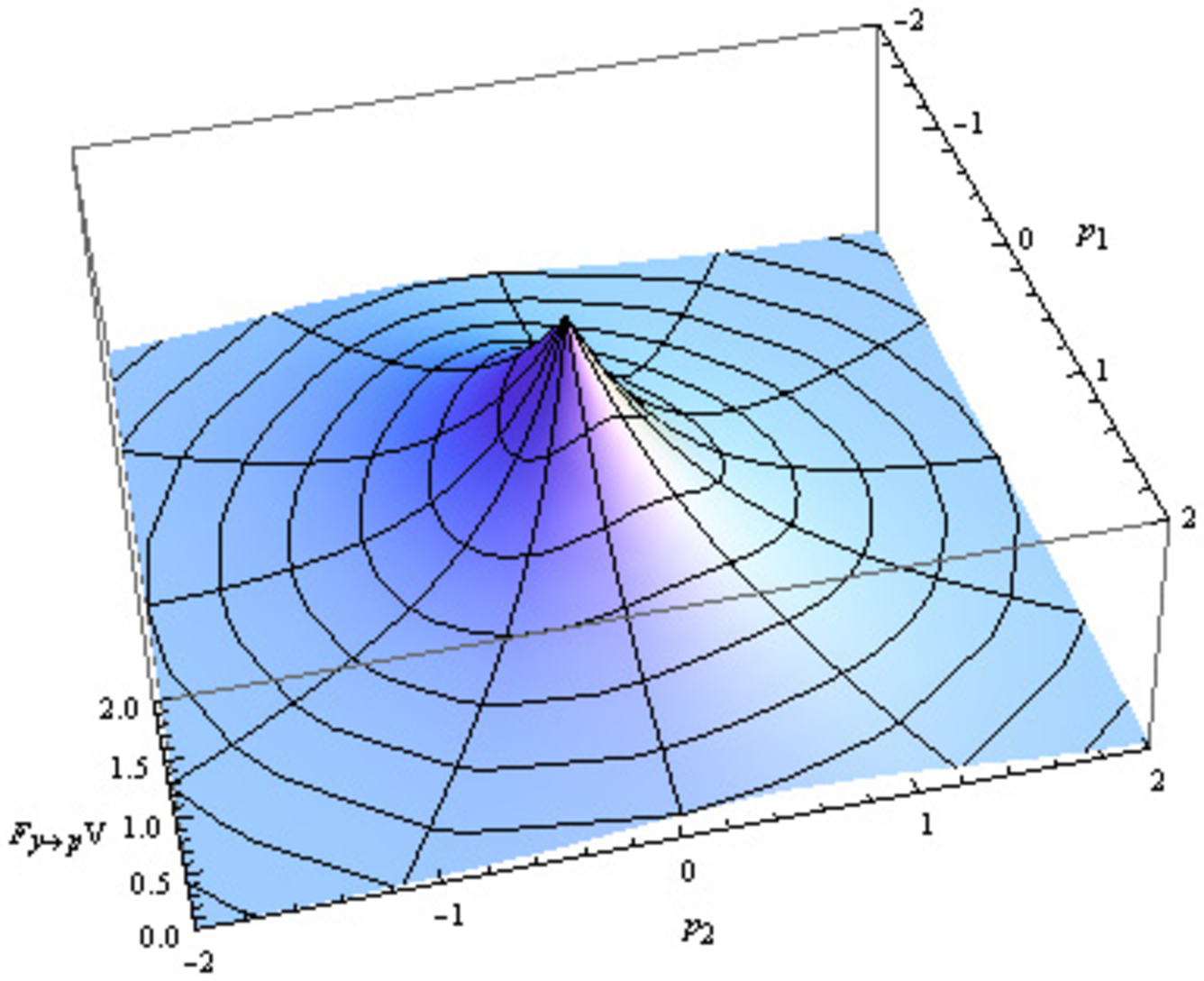}
\caption{The
function $V(y)$ with $b_1=1$ and $b_2=4$ rotated by the angle
$\theta={\pi/10}$ (left) and its Fourier transform $\widetilde V(p)$
(right). \label{Fig_V_tildeV}}
\end{figure}
Let us also give two examples of functions $g_0(t)$
satisfying~\eqref{eq1.1d},
\begin{equation}\label{source}
     \text{(a)}\quad
     g_0(\tau) =ae^{-\tau}(\sin(\alpha\tau+\phi_0)-\sin\phi_0),\qquad
     \text{(b)}\quad g_0(\tau) =e^{-\tau} P(\tau),
\end{equation}
where $\alpha>0$ and $\phi$ are real parameters,
$a=(\alpha^2+1)/(\alpha\cos\phi_0-\alpha^2\sin\phi_0)$ is a
normalizing factor, and $ P(\tau)=\sum_{k=1}^n(k!)^{-1} P_k\tau^k$
is a polynomial of degree $n$ with $\sum_{k=1}^n P_k=1$ (see
Fig.\ref{Fig_g}).
\begin{figure}
\centering
\includegraphics[width=0.48\textwidth,clip]{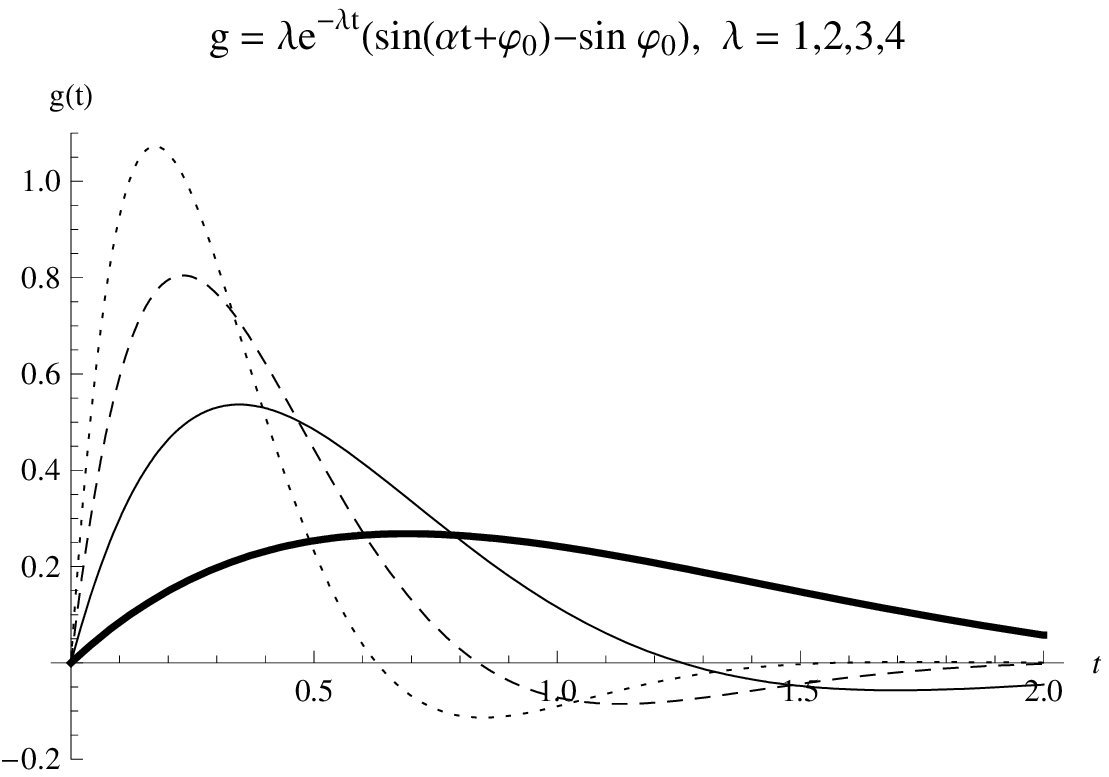}
\includegraphics[width=0.48\textwidth,clip]{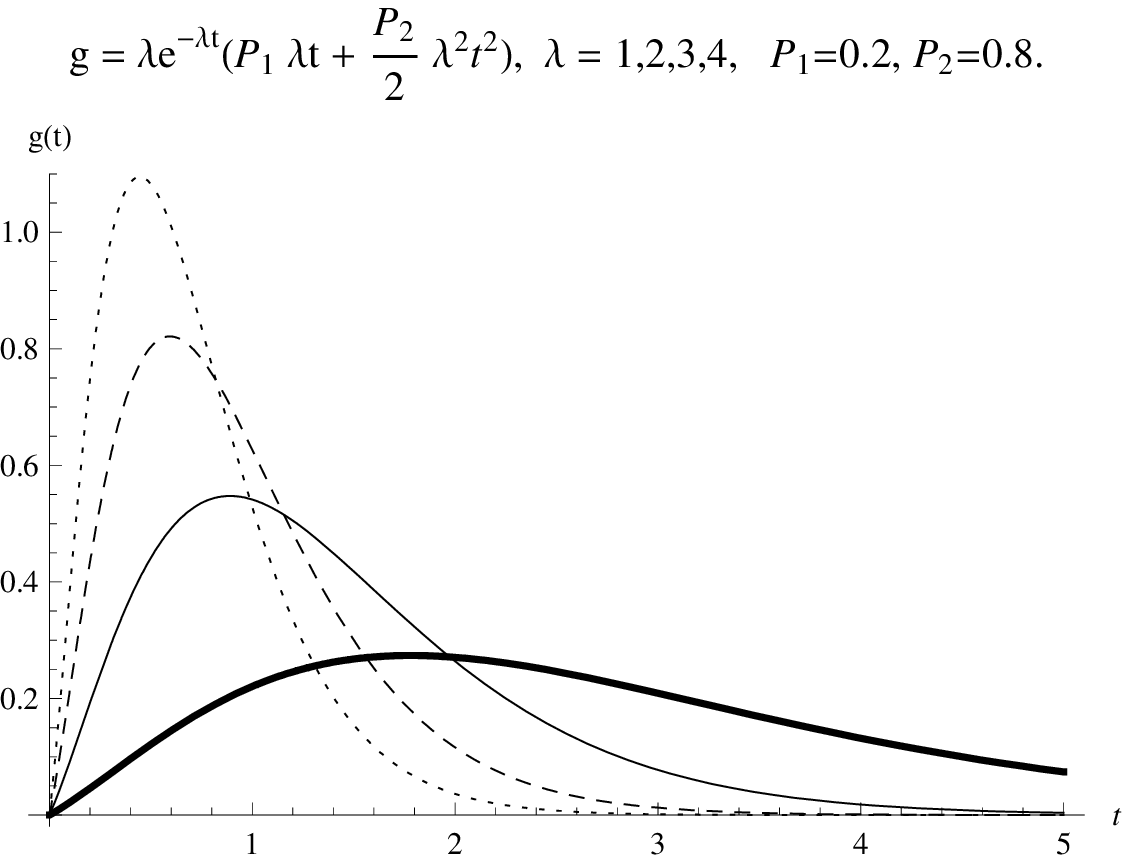}
\caption{Examples of
$g(t)=\lambda g_0(\lambda t)$ for $\lambda=1,2,3,4$:
$g_0(\tau)=e^{-\tau}(\sin(\alpha\tau+\varphi_0)-\sin\varphi_0)$ (upper diagram);
$g_0(\tau)=e^{-\tau}(0.2\tau+0.4\tau^2)$ (lower diagram). \label{Fig_g}}
\end{figure}

\subsection{Operator solution formulas and energy estimates}\label{ss13}

We denote the spatial part of the wave operator in \eqref{eq1.1} by
$L$; thus,
\begin{equation}\label{eq13.1}
Lu=-\frac\partial{\partial x_1}\biggl(c^2(x)\frac{\partial u}{\partial
x_1}\biggr)-\frac\partial{\partial x_2}\biggl(c^2(x)\frac{\partial
u}{\partial x_2}\biggr)\equiv -\langle\nabla,c^2(x)\nabla\rangle u.
\end{equation}
The operator \eqref{eq13.1} (with domain $W_2^2(\mathbf{R}^2)$) is a
nonnegative self-adjoint operator on $L^2(\mathbf R^2)$. Let
$D=\sqrt{L}$ be the positive square root of $L$. Problem
\eqref{eq1.1} becomes
\begin{equation}\label{abst-eq}
    \eta''(t)+D^2\eta(t)=g'(t)v,\qquad
    \eta|_{t=0}=\eta_t|_{t=0}=0.
\end{equation}
Duhamel's formula represents the solution of \eqref{abst-eq} as the
integral
\begin{equation}\label{Duh}
\eta(t)=\int_0^t w(t,\tau)\,d\tau,
\end{equation}
where $w(t,\tau)$ is the solution of the problem\footnote{The
standard Duhamel formula would give $w|_{t=\tau}=0$ and
$w_t|_{t=\tau}=g'(\tau)v$ in \eqref{eq3}, but we have made use of
the special form of the right-hand side.}
\begin{equation}\label{eq3}
w''_{tt}(t,\tau)+D^2 w(t,\tau) =0,\quad w|_{t=\tau}= g(\tau)v,\quad
w_t|_{t=\tau}=0.
\end{equation}
Indeed, the function \eqref{Duh} satisfies \eqref{abst-eq}, because
\begin{align*}
    \eta''(t)+D^2\eta(t)&=\int_0^t\bigl(w''_{tt}(t,\tau)+D^2w(t,\tau)\bigr)\,d\tau
    +\od{}{t}\bigl(w(t,t)\bigr)+w'_t(t,t)\\&=g'(t)v,\qquad
    \eta(0)=0,\qquad \eta'(0)=g(0)v=0
\end{align*}
in view of \eqref{eq1.1d}. Now we can use the general solution
formula (e.g., see \cite[p.~191]{19})
\begin{equation}\label{eq-gsf}
    u(t)=\cos (Dt)u_0+D^{-1}\sin (Dt)u_1
\end{equation}
for the abstract hyperbolic Cauchy problem
\begin{equation}\label{eq-aCp}
    u''(t)+D^2u(t)=0,\qquad u|_{t=0}=u_0,\quad u_t|_{t=0}=u_1
\end{equation}
and write
\begin{equation}\label{eq2.3}
\eta(t)=\biggl[\int_0^t\cos(D(t-\tau))g(\tau)\,d\,\tau\biggr]v
=\re\biggl[\int_0^te^{iD(t-\tau)}g(\tau)\,d\,\tau\biggr]v.
\end{equation}
Here we have used the fact that $g(\tau)$ is real-valued; the real
part of an operator $A$ is defined as usual by $\re
A=\frac12(A+A^*)$. Formula \eqref{eq2.3} is the desired abstract
operator formula for the solution of problem \eqref{eq1.1}.
\begin{remark}
Since the operator $D$ is self-adjoint, it follows that the
expressions $\cos (Dt)$, $\sin (Dt)/D$, and $e^{iDt}$, occurring in
\eqref{eq-gsf} and \eqref{eq2.3}, are well defined in the framework
of functional calculus for self-adjoint operators as functions
$f(D)$ with bounded continuous symbols $f(\xi)=\cos \xi t$,
$f(\xi)=\xi^{-1}\sin \xi t$, and $f(\xi)= e^{i\xi t}$,
respectively. Moreover, $e^{iDt}$ is none other than the strongly
continuous group of unitary operators generated by $D$,
$\re(f(D))=(\re f)(D)$, and, for ``good'' functions $f(\xi)$, the
operator $f(D)$ can be defined not only via the integral over the
spectral measure but also via the Fourier transform as
\begin{equation*}
    f(D)u=\frac1{\sqrt{2\pi}}\langle\widetilde
    f(\tau),e^{i\tau D}u\rangle,\qquad u\in L^2(\mathbf{R}^2),
\end{equation*}
where $\widetilde f$ is the Fourier transform of $f$ and the angle
brackets stand for the value of the distribution $\widetilde f(\tau)$ on
the $L^2(\mathbf{R}^2)$-valued function $e^{i\tau D}u$.
\end{remark}

\begin{remark}\label{rem-2}
The \textit{energy} of the solution of the Cauchy problem
\eqref{eq-aCp} is defined by the formula \cite[p.~191]{19}
\begin{equation}\label{energy-int}
  {\mathcal{E}}[u](t)=\frac12\bigl(\norm{u'(t)}^2+\norm{Du(t)}^2\bigr)
  \equiv \frac12\bigl(\norm{u'(t)}^2+(u(t),Lu(t))\bigr)
\end{equation}
(where $\norm{\,\boldsymbol\cdot\,}$ stands for the $L^2$ norm and
$(\,\boldsymbol\cdot\,,\,\boldsymbol\cdot\,)$ for the $L^2$ inner product) and is conserved in
the course of time. Hence, in view of~\eqref{notat} and the
estimates~\eqref{eq1.1c}, the solution of problem~\eqref{eq3}
(with $\tau$ viewed as a parameter) satisfies
\begin{multline*}
    {\mathcal{E}}[w](t,\tau)
       =\frac{g^2(\tau)}2\norm{D\bigl(V(x/\mu)\bigr)}^2
       =\frac{g^2(\tau)}{2\mu^2}\int_{\mathbf{R}^2}
    \biggl\vert c(x)\nabla V\biggl(\frac x\mu\biggr)\biggr\vert^2\,dx_1dx_2
    \\= \frac{g^2(\tau)}2\int_{\mathbf{R}^2}c^2(\mu y)
    \abs{\nabla V(y)}^2 \,dy_1dy_2
       =\frac{\lambda^2}2g_0^2(\lambda \tau)c_0^2\norm{\nabla V}^2(1+ O(\mu))
\end{multline*}
as $\mu\to0$, where $c_0=c(0)$. Now it follows from \eqref{Duh}
that, with some constant $C$,
\begin{align*}
    {\mathcal{E}}[\eta](t)&\le\left\{\int_0^t\sqrt{{\mathcal{E}}[w](t,\tau)}\,d\tau\right\}^2
    +\frac12\norm{w(t,t)}^2
    \\&\le C\left\{\lambda\int_0^t \abs{g_0(\lambda\tau)}\,d\tau\right\}^2
    +\frac{\mu^2\lambda^2}2g_0^2(\lambda t)\norm{V}^2
    \\&\le C\left\{\int_0^\infty \abs{g_0(\tau)}\,d\tau\right\}^2
    +\frac{\mu^2\lambda^2}2g_0^2(\lambda t)\norm{V}^2
    \\&=O(1)+O(\mu^2\lambda^2 e^{-2\nu\lambda t})
    =O(1)+O(\omega^{-2}e^{-2\nu\lambda t});
\end{align*}
i.e., the energy of the solution is uniformly bounded as $\mu\to 0$
for all $t>\varepsilon>0$. (However, it may have a ``spike'' of the order of
$\omega^{-2}$ for $t\sim 1/\lambda$; of course, this is only important if
$\omega\ll1$.)  In other words, we have chosen a physically natural
normalization of the right-hand side of our problem.
\end{remark}

\begin{remark}\label{rem-3}
For the inhomogeneous wave equation
\begin{equation}\label{rhs}
    u''(t)+D^2u(t)=F(t),
\end{equation}
one has the energy identity
\begin{equation*}
    {\mathcal{E}}[u](t)={\mathcal{E}}[u](0)+\re\int_0^t(F(\tau),u'(\tau))\,d\tau,
\end{equation*}
which implies the well-known estimates\footnote{Their derivation
takes into account the fact that the norm $\norm{u}_s$ is
equivalent to the norm $\norm{(1+L)^{s/2}u}$ by virtue of the
conditions imposed on the velocity $c(x)$.}
\begin{equation}\label{eestiHs}
    \norm{u(t)}_{s+1}+\norm{u'(t)}_{s}\le C(t)\bigl(\norm{u(0)}_{s+1}+\norm{u'(0)}_{s}
    +\sup_{\tau\in[0,t]}\norm{F(\tau)}_s\bigr),
\end{equation}
where $\norm{\,\boldsymbol\cdot\,}_s$ stands for the norm on the Sobolev space
$H^s=W_2^s(\mathbf{R}^2)$; in particular, $\norm{\,\boldsymbol\cdot\,}_0=\norm{\,\boldsymbol\cdot\,}$.
Of the estimates~\eqref{eestiHs}, the most important for us is the
one with $s=0$ (corresponding to the sum of the energy integral and
the $L^2$ norm), in which the main estimates for the norms of
remainders in asymptotic formulas will be obtained. However,
occasionally our argument involves estimates with different~$s$.
\end{remark}

\subsection{Solution splitting into propagating and transient components}\label{ss14}
Let us further transform formula \eqref{eq2.3} to reveal the
structure of the solution and represent it in a form suitable for
subsequent computations. We have
\begin{equation}\label{eq-comp}
\begin{split}
    \int_0^te^{iD(t-\tau)}g(\tau)\,d\,\tau&=
    \int_0^\infty e^{iD(t-\tau)}g(\tau)\,d\,\tau
    -\int_t^\infty e^{iD(t-\tau)}g(\tau)\,d\,\tau\\
    &= e^{iDt}\int_0^\infty e^{-iD\tau}g(\tau)\,d\,\tau
    -\int_0^\infty e^{-iD\tau}g(\tau+t)\,d\,\tau.
\end{split}
\end{equation}
Let $H(\tau)$ be the Heaviside step function ($H(\tau)=1$ for
$\tau\ge0$ and $H(\tau)=0$ for $\tau<0$), and, for $t\ge0$, let
\begin{equation*}
    G(\xi,t)=\int_0^\infty e^{-i\xi\tau}g(\tau+t)\, d\tau\equiv
    \int_{-\infty}^\infty e^{-i\xi\tau}H(\tau)g(\tau+t)\, d\tau
\end{equation*}
be the Fourier transform of $\sqrt{2\pi}H(\tau)g(\tau+t)$ with
respect to the variable $\tau$. (Note that $G(\xi,0)=\sqrt{2\pi}\widetilde
g(\xi)$, where the function $g(\tau)$ is assumed to be extended by
zero for the negative values of $\tau$.) Then formula
\eqref{eq-comp} can be rewritten as
\begin{equation*}
    \int_0^te^{iD(t-\tau)}g(\tau)\,d\,\tau=e^{iDt}G(D,0)-G(D,t)=\sqrt{2\pi}
    e^{iDt}\widetilde g(D)-G(D,t),
\end{equation*}
and accordingly
\begin{align}\label{eq2.5}
\eta(t)&=\eta_{prop}(t)+\eta_{trans}(t),\\
\intertext{where}\label{eq2.5a}
 \eta_{prop}(t)&=\sqrt{2\pi}\re\bigl(e^{iDt}\widetilde g(D)\bigr)v\\
 \nonumber&\equiv
 \sqrt{2\pi}\cos(Dt)\re\widetilde g(D)v - \sqrt{2\pi}\sin(Dt)\im\widetilde g(D)v,\\
 \label{eq2.5b}
 \eta_{trans}(t)&=-\operatorname{Re}(G(D,t))v.
\end{align}

The function $\eta_{prop}(t)$ given by~\eqref{eq2.5a} is the
solution of the Cauchy problem for the homogeneous wave equation
\begin{equation}\label{eq-trans}
    u''(t)+D^2u(t)=0
\end{equation}
with the initial data
\begin{equation}\label{Cauchy-trans}
 u_0=\sqrt{2\pi}\re\widetilde g(D)v,
    \qquad
 u_1=-\sqrt{2\pi}\im\widetilde g(D)Dv.
\end{equation}
(This follows from the comparison of \eqref{eq2.5a} with
\eqref{eq-gsf}.) Hence this function will be called the
\textit{propagating component} of the solution, and the initial
data \eqref{Cauchy-trans} for  $\eta_{prop}(t)$ will be called the
\textit{equivalent source functions}. We shall see in Sec.~\ref{s3}
that, exactly as one should expect, $\eta_{prop}(t)$ propagates
along the characteristics.

The function $\eta_{trans}(t)$ given by~\eqref{eq2.5b} will be
called the \textit{transient component} of the solution, because it
exponentially decays as $\lambda t\to\infty$, as shown by the following
proposition. (We shall also see in Sec.~\ref{ss2.1} that
$\eta_{trans}(t)$ always remains localized near the origin.)
\begin{proposition}\label{pr-exp-decay}
As $\mu\to0$, the propagating component satisfies the estimates
\begin{equation*}
    \norm{\eta_{prop}(t)}_1=O(1),\qquad
    \norm{\eta_{prop}'(t)}=O(1),
\end{equation*}
and the transient component satisfies the estimates
\begin{equation*}
    \norm{\eta_{trans}(t)}_1=O(e^{-\nu\lambda t}),\qquad
    \norm{\eta_{trans}'(t)}=O(\omega^{-1} e^{-\nu\lambda t}),
\end{equation*}
where $\nu$ is the constant in condition~\eqref{eq1.1d}.
\end{proposition}

\begin{proof}
We will estimate the transient part \eqref{eq2.5b} of the solution
directly and the propagating part \eqref{eq2.5a} via the Cauchy
data \eqref{Cauchy-trans} by using the energy estimates. Formulas
\eqref{eq2.5b} and \eqref{Cauchy-trans} involve the real and
imaginary parts of the operator $G(D,t)$ applied to the original
right-hand side source function $v$. (Recall that $\widetilde g(D)$ is a
special case of $G(D,t)$ for $t=0$.) Thus, we need to estimate the
operator $G(D,t)$. Note that, for an arbitrary bounded measurable
function $f(\xi)$, one has
\begin{equation}\label{yes}
    \norm{f(D)\colon H^0\to H^0}\le \sup_{\xi\in\mathbf{R}}\abs{f(\xi)},\quad
    \norm{f(D)\colon H^1\to H^1}\le C\sup_{\xi\in\mathbf{R}}\abs{f(\xi)}
\end{equation}
with some constant $C$ independent of $f$.\footnote{Indeed, the
first estimate is obvious, because the operator $D$ is self-adjoint
on $H^0=L_2(\mathbf{R}^2)$. To obtain the second estimate, we replace the
norm on $H^1$ by the equivalent Hilbert norm
$\norm{u}=(u,(1+L)u)^{1/2}$ (cf.~Remark~\ref{rem-3}); then the
operator $D$ becomes self-adjoint on $H^1$, and the second estimate
follows.} Thus, we need estimates for the function $G(\xi,t)$.
Since $g(\tau)=\lambda g_0(\lambda\tau)$, we have
\begin{equation}\label{G00}
 G(\xi,t)=G_0(\xi/\lambda,\lambda t),
\end{equation}
where
\begin{equation}\label{G0}
 G_0(\xi,t)=\int_0^\infty e^{-i\xi\tau}g_0(\tau+t)\, d\tau
\end{equation}
is the Fourier transform of the function
$\sqrt{2\pi}H(\tau)g_0(\tau+t)$ with respect to $\tau$. By
Lemma~\ref{l1.1} below, we have
\begin{equation*}
    \abs{\sqrt{2\pi}\widetilde g(\xi)}=\abs{G_0(\xi/\lambda,0)}\le C_{00}
\end{equation*}
and hence, by \eqref{Cauchy-trans} and \eqref{yes},
\begin{align*}
 \norm{u_0}_1&=\sqrt{2\pi}\norm{\re\widetilde g(D)v}_1\le C C_{00}\norm{v}_1=O(1),
\\
 \norm{u_1}&=\sqrt{2\pi}\norm{\im\widetilde g(D)Dv}\le C_{00}\norm{Dv}\le
 \widetilde C\norm{v}_1=O(1),
\end{align*}
because $v=V(x/\mu)$ and hence $\norm{v}_1=O(1)$ (cf.~the
computation in Remark~\ref{rem-2}). Now the energy estimates
\eqref{eestiHs} for $s=0$ give the desired estimates for
$\eta_{prop}(t)$. The estimates for the transient part go as
follows, again with the use of Lemma~\ref{l1.1}:
\begin{multline*}
    \norm{\eta_{trans}(t)}_1=\norm{\re G(D,t)v}_1
    \le C\sup_\xi\abs{G_0(\xi/\lambda,\lambda t)}\norm{v}_1
    \\\shoveright{\le
    CC_{00}e^{-\nu\lambda t}\norm{v}_1=O(e^{-\nu\lambda t}),}
\\
    \shoveleft{\norm{\eta_{trans}'(t)}=\norm{\re \pd{G}{t}(D,t)v}
    \le \sup_\xi\left\vert\lambda \pd{G_0}{t}\biggl(\frac\xi\lambda,\lambda
    t\biggr)\right\vert\norm{v}}
    \\\le
     C_{01}\lambda e^{-\nu\lambda t}\norm{v}=O(\mu\lambda e^{-\nu\lambda t})=O(\omega^{-1} e^{-\nu\lambda t}),
\end{multline*}
because $v=V(x/\mu)$ and hence $\norm{v}=O(\mu)$. This completes
the proof.
\end{proof}

The following lemma establishes the estimates for $G_0(\xi,t)$ used
in the proof given above and also estimates that will be useful
below.

\begin{lemma}\label{l1.1}
The function $G_0(\xi,t)$ satisfies the estimates
\begin{equation}\label{esti-g0}
    \left\vert\pd{^{m+k}G_0}{t^m\partial \xi^k}(\xi,t)\right\vert\le C_{km}e^{-\nu
    t}(1+\abs{\xi})^{-k-1},\qquad k,m=0,1,2,\dots\,,
\end{equation}
with some constants $C_{km}$. For $t=0$ and $m=0$, one has the
better estimates
\begin{equation}\label{esti-g00}
    \left\vert\pd{^kG_0}{\xi^k}(\xi,0)\right\vert\le C_{k0}(1+\abs{\xi})^{-k-2},\qquad
    k=0,1,2,\dots\,.
\end{equation}
\end{lemma}

\begin{proof}
First, let us prove the estimates \eqref{esti-g0} and
\eqref{esti-g00} for $\abs{\xi}\le1$.  Then we have
\begin{equation*}
    \left\vert\pd{^{m+k}G_0}{t^m\partial \xi^k}(\xi,t)\right\vert
    =\left\vert\int_0^\infty (-i\tau)^ke^{-i\xi\tau}g_0^{(m)}(\tau+t)\, d\tau\right\vert
    \le C_m e^{-\nu t}\int_0^\infty\tau^ke^{-\nu\tau}\, d\tau
\end{equation*}
by \eqref{eq1.1d}, whence the claim follows. Now let $\abs{\xi}>1$.
Then we write
\begin{equation*}
    \pd{^mG_0}{t^m}(\xi,t)=
    \biggl(\frac i\xi\biggr)^N\int_0^\infty \biggl[\od{^N}{\tau^N}
    \bigl(e^{-i\xi\tau}\bigr)\biggr]g_0^{(m)}(\tau+t)\,
    d\tau
\end{equation*}
for some integer $N>k+1$ and then integrate by parts $N$ times,
thus obtaining
\begin{equation}\label{hmmm}
    \pd{^mG_0}{t^m}(\xi,t)=\sum_{l=1}^N(i\xi)^{-l}g_0^{(m+l-1)}(t)+(i\xi)^{-N}\int_0^\infty
    e^{-i\xi\tau}g_0^{(m+N)}(\tau+t)\, d\tau.
\end{equation}
Next, we differentiate both sides of \eqref{hmmm} $k$ times with
respect to $\xi$, which gives
\begin{multline*}
    \pd{^{m+k}G_0}{t^m\partial \xi^k}(\xi,t)
    =i^{-k}\sum_{l=1}^N\frac{(l+k-1)!}{(l-1)!}(i\xi)^{-l-k}g_0^{(m+l-1)}(t)
    \\
    {}+i^{-k}\sum_{s=0}^k\binom ks\frac{(l+s-1)!}{(l-1)!}(i\xi)^{-N-s}
    \int_0^\infty\tau^{k-s}
    e^{-i\xi\tau}g_0^{(m+N)}(\tau+t)\, d\tau.
\end{multline*}
Here all factors $g_0^{(m+l-1)}(t)$ and the integral are bounded in
modulus by $\const\,\boldsymbol\cdot\, e^{-\nu t}$ by virtue of~\eqref{eq1.1d},
and the smallest power of $\xi^{-1}$ on the right-hand side is
$\xi^{-k-1}$, which implies the estimate \eqref{esti-g0}. For $t=0$
and $m=0$, the smallest power of $\xi^{-1}$ on the right-hand side
is $\xi^{-k-2}$, since $g_0(0)=0$, and we have the
estimate~\eqref{esti-g00}.
\end{proof}

\section{Asymptotics of the solution}\label{s2}

In this section, we describe the asymptotics as $\mu\to0$ of the
solution $\eta(t)=\eta_{prop}(t)+\eta_{trans}(t)$ of problem
\eqref{eq1.1}, \eqref{eq1.2}. In all theorems in this section, we
assume that all conditions stated in Sec.~\ref{ss11} are satisfied.
Recall that the problem also contains the large parameter $\lambda$,
which is related to $\mu$ by the condition $\omega<\omega_0$ (see
\eqref{omega1}), where $\omega=c_0(\lambda\mu)^{-1}$ (see \eqref{omega}).
If $\omega$ can be treated as a second small parameter (i.e., the
distance traveled by the waves in the lifetime of the source is
much smaller than the source diameter), then additional Taylor
series expansions in $\omega$ lead to further simplifications in the
asymptotic formulas.

\subsection{Asymptotics of the transient component}\label{ss2.1}

The asymptotics of the transient component $\eta_{trans}(t)$ of the
solution as $\mu\to0$ is given by the following theorem.

\begin{theorem}\label{thm1}
One has
\begin{equation}\label{sol_tran}
    \eta_{trans}(x,t) = - \frac 1 {2\pi} \iint
\re G_0 (\omega|p|,\lambda t ) \widetilde V(p) e^{i p x/\mu} dp_1\,dp_2 + R(t),
\end{equation}
or, in the polar coordinates $(r,\varphi)$, $x = r \mathbf n(\varphi)$,
where $\mathbf n(\varphi)=(\cos\varphi,\sin\varphi)$,
\begin{multline}\label{sol_tran_polar}
\eta_{trans}(r \mathbf n (\varphi)) = - \frac 1 {2\pi}
\int_0^{2\pi}\!\!\!\! \int_0^\infty \rho \re G_0 (\omega\rho,\lambda t)
\\
\times \widetilde V\bigl(\rho \mathbf n(\psi)\bigr) e^{i r \rho
\cos(\psi - \varphi)/\mu} d\rho\, d\psi + R(t),
\end{multline}
where the remainder $R(t)$ satisfies the estimates
\begin{equation}\label{eq-as1b}
    \norm{R(t)}_1=O(\mu e^{-\nu\lambda t}),\qquad
    \norm{R'(t)}=O(\mu\omega^{-1} e^{-\nu\lambda t}),\qquad \mu\to0.
\end{equation}
\end{theorem}

\begin{proof}
Consider the operators
\begin{equation}\label{freezeoper}
    L^{(0)}=-c_0^2\nabla^2,\qquad D^{(0)}=(L^{(0)})^{1/2}.
\end{equation}
Thus, $L^{(0)}$ is obtained by freezing the coefficients of the
operator $L$ at the origin, and $D^{(0)}$ is just the positive
square root of the positive self-adjoint operator $L^{(0)}$.
\begin{lemma}\label{l2.1}
One has
\begin{equation}\label{eql2.1}
    \bigl[\re G(D^{(0)},t)-\re G(D,t)\bigr] V\biggl(\frac x\mu\biggr)=R(t),
\end{equation}
where $R(t)$ satisfies the estimates \eqref{eq-as1b}.
\end{lemma}
The proof of this lemma will be given in Sec.~\ref{s4}. Thus, the
operator $D$ in the expression \eqref{eq2.5b} for
$\eta_{trans}(x,t)$ can be replaced by the operator $D^{(0)}$ with
constant coefficients. Now it remains to compute $\re
G(D^{(0)},t)V(x/\mu)$. Since $D^{(0)}$ is an operator with constant
coefficients and with symbol $c_0\abs{p}$, it follows that
\begin{equation}\label{mind-it}
    F(D^{(0)})={\mathcal{F}}^{-1}\circ F(c_0\abs{p})\circ{\mathcal{F}}
\end{equation}
for any function $F(\xi)$, where ${\mathcal{F}}$ is the Fourier transform and
the middle factor on the right-hand side is the operator of
multiplication by $F(c_0\abs{p})$. Thus we obtain
\begin{equation*}
    G(D^{(0)},t)V\biggl(\frac x\mu\biggr)=\frac 1 {2\pi} \iint
\re G_0(\omega|p|,\lambda t) \widetilde V(p) e^{i p x/\mu} dp_1\,dp_2
\end{equation*}
(we have used the formula for $G(\xi,t)$ and made obvious changes
of variables), which proves Theorem~\ref{thm1}.
\end{proof}

\subsection{Asymptotics of the equivalent source functions}

Now we proceed to the computation of the propagating component of
the solution. It satisfies the Cauchy problem \eqref{eq-trans},
\eqref{Cauchy-trans}, and so a good starting point would be to
compute the equivalent source functions \eqref{Cauchy-trans}. Once
we compute them (asymptotically) and prove that they are localized
near the origin, we can use the methods developed in
\cite{5,6,24,7,8,12} to obtain the asymptotics of the propagating
solution component. However, we would like to apply ready-to-use
formulas from these papers rather than to write out new formulas
based on the same ideas. The formulas in \cite{5,6,24,7,8,12} were
obtained for the case in which $u_1$, the initial data for the
$t$-derivative of the solution, is zero. So we resort to the
following trick.
\begin{proposition}
The propagating solution component  $\eta_{prop}(t)$ can be
represented in the form
\begin{equation}\label{eq2-13-0}
    \eta_{prop}(t)=\eta_1(t)+\eta_2'(t),
\end{equation}
where $\eta_1(t)$ and $\eta_2(t)$ are the solutions of the Cauchy
problems
\begin{align}
\label{eq2.13}
   \eta_1''(t)+D^2\eta_1&=0, \qquad
   \eta_1|_{t=0}=\sqrt{2\pi}\re\widetilde g(D)v,\quad\eta_1'|_{t=0}=0, \\
 \label{eq2.14}
   \eta_2''(t)+D^2\eta_2&=0, \qquad
   \eta_2|_{t=0}=\sqrt{2\pi}D^{-1}\im\widetilde g(D)v,\quad\eta_2'|_{t=0}=0.
\end{align}
\end{proposition}
\begin{proof}
The sum \eqref{eq2-13-0} obviously satisfies the wave equation
\eqref{eq-trans}. Next,
\begin{equation*}
    \eta_2'|_{t=0}=0, \qquad
    (\eta_2')'|_{t=0}=\eta_2''|_{t=0}=-D^2\eta_2|_{t=0}
    =-\sqrt{2\pi}D\im\widetilde g(D)v,
\end{equation*}
which shows that the initial conditions \eqref{Cauchy-trans} are
satisfied and hence completes the proof.
\end{proof}
\noindent Thus, let us compute the asymptotics of the new
equivalent source functions
\begin{equation}\label{eqsotraf0}
    \eta_{10}=\sqrt{2\pi}\re\widetilde g(D)v,\qquad \eta_{20}=\sqrt{2\pi}D^{-1}\im\widetilde
    g(D)v.
\end{equation}

\begin{theorem}\label{th-asf2}
The equivalent source functions \eqref{eqsotraf0} have the
following as\-ymp\-to\-tics as $\mu\to0$\textup:
\begin{equation}\label{eqsotraf-b}
    \eta_{10}=U_1\biggl(\frac x\mu\biggr)+R_1,
    \qquad \eta_{20}=U_2\biggl(\frac x\mu\biggr)+R_2,
\end{equation}
where the Fourier transforms of the functions $U_1(y)$ and $U_2(y)$
are given by the formulas
\begin{equation}\label{eqsotraf-c}
    \widetilde U_1(p) =\sqrt{2\pi}\re\widetilde g_0(\omega\abs{p})\widetilde
    V(p),\qquad
    \widetilde U_2(p) =\sqrt{2\pi} \lambda^{-1}
    \im\frac{\widetilde g_0(\omega\abs{p})}{\omega\abs{p}}\widetilde V(p)
\end{equation}
and the remainders satisfy the estimates
\begin{equation}\label{esti-z}
    \norm{R_1}_1=O(\mu),\qquad
    \norm{R_2}_2=O(\mu).
\end{equation}
\end{theorem}
\begin{proof}
The proof goes along the same lines as that of Theorem~\ref{thm1}.
Namely, we prove that the operator $D$ in
formulas~\eqref{eqsotraf0} can asymptotically be replaced by
$D^{(0)}$ and then compute the Fourier transforms of $U_1$ and
$U_2$ using formula~\eqref{mind-it}. The latter computation is
trivial, and we omit it altogether. As for the first part, it is
given by the following lemma, which will be proved together with
Lemma~\ref{l2.1} in Sec.~\ref{s4}.
\begin{lemma}\label{l2.2}
One has
\begin{equation}\label{eql2.2}
\begin{aligned}
    &\sqrt{2\pi}\bigl[\re\widetilde g(D)-\re\widetilde g(D^{(0)})\bigr] V\biggl(\frac
    x\mu\biggr)=R_1,\\
    &\sqrt{2\pi}\bigl[D^{-1}\im\widetilde g(D)-(D^{(0)})^{-1}\im\widetilde g(D^{(0)})\bigr]
    V\biggl(\frac x\mu\biggr)=R_2,
\end{aligned}
\end{equation}
where $R_1$ and $R_2$ satisfy the estimates \eqref{esti-z}.
\end{lemma}
\noindent This completes the proof of Theorem~\ref{th-asf2}.
\end{proof}
\begin{remark}\label{rem-4}
If we replace $\eta_{10}$ and $\eta_{20}$ in the Cauchy problems
for $\eta_1$ and $\eta_2$ by $U_1(x/\mu)$ and $U_2(x/\mu)$,
respectively, then the resulting error $\delta(t)$ in the computation
of $\eta_{prop}(t)$ will satisfy the estimates
\begin{equation*}
    \norm{\delta(t)}_1=O(\mu),\qquad \norm{\delta'(t)}=O(\mu),\qquad
    \mu\to0,
\end{equation*}
uniformly on any finite time interval. Indeed, let us write
$\delta(t)=\delta_1(t)+\delta_2'(t)$, where $\delta_1$ and $\delta_2$ are the
errors in $\eta_1$ and $\eta_2$, respectively. Then, by virtue of
the energy estimates \eqref{eestiHs}, we have
\begin{align*}
    \norm{\delta_1(t)}_1+\norm{\delta_1'(t)}&\le C(t)\norm{\delta_1(0)}_1=C(t)\norm{R_1}_1=O(\mu),
    \\
    \norm{\delta_2'(t)}_1+\norm{\delta_2''(t)}&\le C(t)\norm{\delta_2''(0)}
    =C(t)\norm{D^2\delta_2(0)}\\&=C(t)\norm{D^2R_2}\le C_1(t)\norm{R_2}_2
    =O(\mu).
\end{align*}
Thus, the accuracy provided by Theorem~\ref{th-asf2} permits
computing the propagating part of the solution modulo $O(\mu)$ in
the energy norm.
\end{remark}

\subsection{Asymptotics of the propagating part}\label{s3}

Remark~\ref{rem-4} shows that, to compute the asymptotics of the
propagating part of the solution of problem \eqref{eq1.1} modulo
$O(\mu)$ in the energy norm, it suffices to solve problems
\eqref{eq2.13} and \eqref{eq2.14} asymptotically with the initial
data replaced by the functions $U_1(x/\mu)$ and $U_2(x/\mu)$
indicated in Theorem~\ref{th-asf2}. Thus, we need to solve the
problems
\begin{align}
\label{eq2.13-a}
   \eta_1''(t)+D^2\eta_1&=0, \qquad
   \eta_1|_{t=0}=U_1(x/\mu),\quad \eta_1'|_{t=0}=0, \\
 \label{eq2.14-a}
   \eta_2''(t)+D^2\eta_2&=0, \qquad
   \eta_2|_{t=0}=U_2(x/\mu),\quad \eta_2'|_{t=0}=0.
\end{align}
(We denote the new unknown functions by the same letters
$\eta_{1,2}$; this will not lead to a misunderstanding.) The
initial data in these problems are localized near the origin, and
hence the asymptotics of solutions of these problems modulo
$O(\mu)$ in all spaces $H^s$ can be obtained with the use of the
approach developed in \cite{5,7,8,24} and based on the Maslov
canonical operator \cite{13,14}. Let us briefly recall this
construction.

\subsubsection{Bicharacteristics, canonical operator and solution formulas}
In the phase space $\mathbf{R}^4_{x,p}$ with the coordinates
$(x,p)=(x_1,x_2,p_1,p_2)$, consider the Hamiltonian system
\begin{equation}\label{eq3.5}
\dot p=-\frac{\partial\mathcal{H}}{\partial x}, \qquad \dot
x=\frac{\partial \mathcal{H}}{\partial p}
\end{equation}
corresponding to the Hamiltonian function $\mathcal{H}=|p|c(x)$.
This system determines the Hamiltonian phase flow
$g_{\mathcal{H}}^t$. Let $\mathbf n(\psi)=
{}^t(\cos\psi,\sin\psi)$. Consider the Lagrangian manifold
$\Lambda_0=\{p=\mathbf n(\psi), x=\alpha\mathbf n(\psi)\},$
isomorphic to the two-dim\-en\-sion\-al cylinder, where $\psi\in[0,2\pi)$
and $\alpha\in\mathbf R$ are coordinates on $\Lambda_0$. By
shifting this manifold along the flow $g_{\mathcal{H}}^t,$ we
obtain the family of Lagrangian manifolds
$\Lambda_t=g_{\mathcal{H}}^t\Lambda_0$, each of which is equipped
with the same coordinate system $(\psi,\alpha)$ as $\Lambda_0$. We
take the point with coordinates $(\psi,\alpha)=(0,0)$ for the
distinguished point on $\Lambda_0$ and construct the Maslov
canonical operator $K^{h}_{\Lambda_t}$ \cite{13,14} on each of the
manifolds $\Lambda_t$. (Here $h\to0$ is the small parameter
occurring in the construction of the canonical operator; all
Jacobians in the definition of $K^{h}_{\Lambda_t}$ are taken with
respect to the coordinates $(\psi,\alpha)$.)

It follows from the results in \cite{7,8,24,5} that the asymptotics
of the solutions $\eta_{1,2}$ of problems \eqref{eq2.13-a} and
\eqref{eq2.14-a} can be obtained as follows. Using the Fourier
transforms \eqref{eqsotraf-c} of the equivalent source functions
computed in Theorem~\ref{th-asf2}, we introduce the following two
smooth functions on $\Lambda_t$, independent of $t$ and $\alpha$
but depending on the coordinate $\psi$ and an additional parameter
$\rho$:
\begin{align*}
\varphi_1(\psi,\rho)&=\widetilde U_1(\rho \mathbf
n(\psi))=\sqrt{2\pi}\re\widetilde g_0(\omega\rho)\widetilde V(\rho \mathbf n(\psi)), \\
\varphi_2(\psi,\rho)&=\widetilde U_2(\rho \mathbf
n(\psi))=\sqrt{2\pi}\lambda^{-1} \im\frac{\widetilde g_0(\omega\rho)}{\omega\rho}\widetilde
V(\rho\mathbf n(\psi)).
\end{align*}
Then the formulas in \cite{5,8} give
\begin{equation}
\eta_{1,2}(t)=\sqrt{\frac{\mu}{2\pi}}\operatorname{Re}\biggl(e^{-i\pi/4}\int_0^\infty
K^{\mu/\rho}_{\Lambda_t} (\sqrt\rho\,\varphi_{1,2}(\psi,
\rho))\,d\,\rho\biggr)+O(\mu). \label{eq3.6}
\end{equation}
Let us find the derivative $\eta_2'(t)$. By the commutation formula
\cite{13} for the canonical operator, we have
\begin{multline*}
\eta_2'(t)=\sqrt{\frac{\mu}{2\pi}}\operatorname{
Re}\biggl(e^{-i\pi/4}\int_0^\infty \frac{\partial}{\partial t}
K^{\mu/\rho}_{\Lambda_t} (\sqrt\rho\,\varphi_2(\psi, \rho))\,d
\rho\biggr)
 \\
 =\sqrt{\frac{\mu}{2\pi}}\operatorname{
Re}\Bigl(e^{-i\pi/4}\int_0^\infty K^{\mu/\rho}_{\Lambda_t}
\Big[-\frac{i\rho}\mu\mathcal{H}\big|_{\Lambda_t}\sqrt\rho\,\varphi_2(\psi,
\rho)\Big]\,d \rho\Bigr)+O(\mu).
\end{multline*}
But the Hamiltonian $\mathcal{H}$ is preserved along the
trajectories of the Hamiltonian system, and hence
$\mathcal{H}\big|_{\Lambda_t}=\mathcal{H}\big|_{\Lambda_0}=c(\alpha\mathbf
n(\psi))$. It was shown in~\cite{5} that, modulo lower-order terms,
one can set $\alpha=0$ in the functions on $\Lambda_t$. Taking into
account the definition of $\varphi_2$, we obtain
\begin{equation*}
\eta_2'(t)=\sqrt\mu\operatorname{
Re}\biggl(e^{-\frac{i\pi}4}\int_0^\infty K^{\mu/\rho}_{\Lambda_t}
\bigl[\sqrt\rho\bigl(-i\operatorname{Im}\widetilde g_0(\omega\rho) \widetilde
V(\rho\mathbf n(\psi))\bigr)\bigr]\,d\,\rho\biggr)+O(\mu).
\end{equation*}
Finally, we use the formula $\eta_{prop}(t)=\eta_1(t)+\eta_2'(t)$
and arrive at the following theorem.
\begin{theorem}\label{thm3}
The propagating part of the solution has the following
asymptotics\textup:
\begin{equation}\label{eq3.7}
\eta_{prop}(t)=\sqrt\mu\operatorname{
Re}\biggl(e^{-\frac{i\pi}4}\int_0^\infty K^{\mu/\rho}_{\Lambda_t}
\bigl[\sqrt\rho\,\overline{{\widetilde g}_0}\,\bigl(\omega\rho\bigr) \widetilde
V(\rho\mathbf n(\psi))\bigr]\,d\,\rho\biggr)+R(t),
\end{equation}
where the bar stands for complex conjugation and the remainder
satisfies the estimates
\begin{equation}\label{eq3.7a}
    \norm{R(t)}_1=O(\mu),\qquad \norm{R'(t)}=O(\mu)
\end{equation}
uniformly on any finite interval of time $t$.
\end{theorem}

\subsubsection{Asymptotics near the front}

Now let us compute the propagating part \eqref{eq3.7} of the
solution in more explicit terms. To this end, we need some
geometry. Let $(P(t,\psi),X(t,\psi))$, $\psi\in[0.2\pi)$, be the
family of solutions of the Hamiltonian system \eqref{eq3.5} with
the initial conditions
\begin{equation}\label{4.1}
  p|_{t=0}=\mathbf n(\psi), \quad x|_{t=0}=0.
\end{equation}
For each $t,$ the equations $p=P(t,\psi)$, $x=X(t,\psi)$,
$\psi\in[0.2\pi)$, define a smooth closed curve $\Gamma_t$ in the
four-dimensional phase space $\mathbf{R}^4_{x,p}$; this curve is called
the \textit{wave front} in $\mathbf{R}^4_{x,p}$. The projection
$\gamma_t=\{x=X(t,\psi)\colon \psi\in[0.2\pi)\}$ of $\Gamma_t$ into
$\mathbf{R}^2_x$ is called the \textit{front in the configuration space}.
In contrast to $\Gamma_t$, the curve $\gamma_t$ may well be
nonsmooth; namely, it may have turning (or focal) points (in~this
case, $X_\psi=0$ for some $\psi$) and points of self-intersection.
Moreover, the front $\gamma_0$ at the initial time $t=0$ is just
the point $x=0$.

For each $t$, the function \eqref{eq3.7} is localized in a
neighborhood of the front $\gamma_t$ \cite{5,6,24,7,8,12}. Formula
\eqref{eq3.7} provides the global asymptotics of the propagating
part of the solution; i.e., this formula holds both near regular
and near focal points of the front. The formula can be simplified
in a neighborhood of any point of the front, but the simplified
expression depends on whether the point is regular or focal. Here
we restrict ourselves to the case of a neighborhood of a regular
point.

Take some time $t$ and angle $\psi^0$ and assume that the point
$X(t,\psi^0)\in\gamma_t$ is not focal; i.e., $X_\psi(t,\psi^0)\ne
0$. In some neighborhood of $X(t,\psi^0)$, we can introduce the
local coordinates $(\psi,y)$, where $y=y(x,t)$ is the (signed)
distance between the point $x$ and the front and $\psi=\psi(x,t)$
is determined by the condition that the vector $x-X(t,\psi(x,t))$
is orthogonal to the vector tangent to the wave front at the point
$X(t,\psi(x,t))$; in other words
\begin{align*}
 &\langle x-X(t,\psi(x,t)), X_\psi(t,\psi(x,t))\rangle=0.
\\
\intertext{Set}
  &S(x,t)=\langle P(t,\psi(x,t)),x-X(t,\psi(x,t))\rangle.
\end{align*}
Next, we introduce the \textit{Morse index} $m(t,\psi^0)$ of the
trajectory $X(\tau,\psi^0)$, $\tau\in(0,t]$, which is the number of
zeros of the function $|X_\psi(\tau,\psi^0)|$ on the half-open
interval $\tau\in(0,t]$ \cite{13}.

It may happen that some region of points $x$ where we intend to
write out the asymptotics simultaneously belongs to several
neighborhoods of the above-mentioned type, where the corresponding
points $X(t,\psi^0)$ lie on several (but finitely many!) distinct
arcs of the front $\gamma_t$. (For example, this is the case if we
study the asymptotics near a point of self-intersection of the
front $\gamma_t$.) Then all these arcs contribute to the asymptotics
at such points $x$, and we use an additional subscript $j$ to
distinguish these neighborhoods as well as all associated objects
($\psi^0$, $\psi(x,t)$, $S(x,t)$, Morse index, etc.). Now from the
results in \cite{5,6,7,12} we obtain the following theorem.
\begin{theorem}\label{p4.1}
In a neighborhood of the front $\gamma_t$ but outside a neighborhood
of the focal points, the asymptotic formula~\eqref{eq3.7} for the
propagating part of the solution can be rewritten in the form
\begin{multline}\label{4.2}
\eta_{prop}(t) = \sqrt\mu\re\sum_j\biggl[\frac{ e^{-i\pi
m(\psi^0_j,t)/2}}{\sqrt{|X_\psi(\psi,t)|}}
\sqrt{{\frac{c_0}{c(X(\psi,t))}}} F\biggl(\frac{S_j(x,t)}\mu,\psi\biggr)
\biggr]_{\psi=\psi_j(x,t)}\\{}+R(t),
\end{multline}
where
\begin{equation}\label{4.3}
F(z,\psi)=e^{-i\pi/4}\int_0^\infty \sqrt\rho\,\bar{\widetilde
g}_0(\omega\rho)\widetilde V(\rho\mathbf n(\psi))e^{iz\rho}d\rho,
\end{equation}
$R(t)$ satisfies the estimate \eqref{eq3.7a}, and the sum with
respect to $j$ is taken over all distinct arcs of $\gamma_t$
contributing to the asymptotics at $x$.\footnote{More formally, for
example, fix an $\varepsilon>0$; the intersection of $\gamma_t$ with the
$\varepsilon$-neighborhood of $x$ can be covered by finitely many arcs of
length $\le\varepsilon$; take the contributions of all these arcs.}
\end{theorem}
\begin{remark}\label{rem-1}
The factor
\begin{equation*}
  \frac1{\sqrt{|X_\psi(\psi,t)|}} \sqrt{{\frac{c_0}{c(X(\psi,t))}}}
\end{equation*}
includes the two-dimensional analog of the so-called Green law and
the trajectory divergence related to the velocity $c(x)$ (with
height $c^2(x)$ describing the bottom topography). The function $F$
depends on the time and space shape of the source generating the
waves \cite{5,7,6,12,8,24}. Formulas \eqref{eq3.7}, \eqref{4.2},
and \eqref{4.3} apply to any localized perturbation.
\end{remark}

\section{Obtaining
asymptotic expansions by noncommutative analysis}\label{s4}

The aim of this section is to prove Lemmas~\ref{l2.1}
and~\ref{l2.2}. Vaguely speaking, these lemmas state that the
replacement of the operator $D$ by the operator $D^{(0)}$
\eqref{freezeoper} with constant coefficients in certain
expressions results in an $O(\mu)$ error. However, it is much
easier to deal with functions of the differential operators $L$ and
$L^{(0)}$ than with functions of their square roots, the
pseudodifferential operators $D=\sqrt{L}$ and
$D^{(0)}=\sqrt{L^{(0)}}$. Hence in Sec.~\ref{sec-prelim} we
represent the latter functions via the former and accordingly
restate the lemmas. In Sec.~\ref{ss21}, we make all noncommutative
computations.

\subsection{Eliminating the square roots}\label{sec-prelim}

\begin{lemma}\label{l44}
The functions $\re G_0(\xi,t)$ and $\xi^{-1}\im G_0(\xi,t)$ are
smooth even functions of $\xi$ and hence smooth functions of
$\xi^2$.
\end{lemma}

\begin{proof}
The function $g_0(\tau)$ is real-valued, and
$\overline{G_0(\xi,t)}=G_0(-\xi,t)$ by~\eqref{G0}. Thus, $\re
G_0(-\xi,t)=\re G_0(\xi,t)$ and $\im G_0(-\xi,t)=-\im G_0(\xi,t)$;
i.e., the real part of $G_0$ is an even function of $\xi$, and the
imaginary part of $G_0$ is an odd function of $\xi$. Hence the
desired claim follows.
\end{proof}

Now let us introduce the functions
\begin{equation}\label{fofxi}
\begin{gathered}
    f_1(\xi)=\re G_0(\xi^{1/2},0),\qquad
    f_2(\xi)=\xi^{-1/2}\im G_0(\xi^{1/2},0),\\
    f_3(\xi,t)=\re G_0(\xi^{1/2},t).
\end{gathered}
\end{equation}
By Lemma~\ref{l44}, these functions are smooth for all $\xi$,
including $\xi=0$. Formulas \eqref{eqsotraf0} for the equivalent
sources and \eqref{eq2.5b} for the transient solution component can
now be rewritten as
\begin{equation}\label{eqsotraf}
\begin{gathered}
    \eta_{10}=f_1(\lambda^{-2}L)V\biggl(\frac x\mu\biggr),
    \qquad \eta_{20}=\lambda^{-1}f_2(\lambda^{-2}L)V\biggl(\frac x\mu\biggr),
    \\
    \eta_{trans}(t)=-f_3(\lambda^{-2}L,\lambda t)V\biggl(\frac x\mu\biggr).
\end{gathered}
\end{equation}
Indeed, for example,
\begin{multline*}
    \lambda^{-1}f_2(\lambda^{-2}L)=\lambda^{-1}(\lambda^{-2}L)^{-1/2}\im
    G_0\bigl((\lambda^{-2}L)^{1/2},0\bigr)\\=D^{-1}\im G_0(\lambda^{-1}D,0)
    =D^{-1}\im G(D,0)=\sqrt{2\pi}D^{-1}\im \widetilde g(D).
\end{multline*}
The following theorem is an equivalent restatement of
Lemmas~\ref{l2.1} and~\ref{l2.2} in terms of functions of $L$ and
$L^{(0)}$. (We write $R_3(t)=-R(t)$ to unify the notation.)
\begin{theorem}\label{th-asf}
One has
\begin{equation}\label{eqsotraf-a}
\begin{aligned}
    f_1(\lambda^{-2}L)V\biggl(\frac x\mu\biggr)&=f_1(\lambda^{-2}L^{(0)})V\biggl(\frac x\mu\biggr)+R_1,
    \\
    \lambda^{-1}f_2(\lambda^{-2}L)V\biggl(\frac x\mu\biggr)&=\lambda^{-1}f_2(\lambda^{-2}L^{(0)})V\biggl(\frac x\mu\biggr)+R_2,
    \\
    f_3(\lambda^{-2}L,\lambda t)V\biggl(\frac x\mu\biggr)&=f_3(\lambda^{-2}L^{(0)},\lambda t)V\biggl(\frac x\mu\biggr)+R_3(t),
\end{aligned}
\end{equation}
where the remainders satisfy the estimate
\begin{equation}\label{esti-a}
\begin{gathered}
    \norm{R_1}_1=O(\mu),\qquad
    \norm{R_2}_2=O(\mu),\\
    \norm{R_3(t)}_1=O(\mu e^{-\nu\lambda t}),\qquad
    \norm{R_3'(t)}=O(\mu^2\lambda e^{-\nu\lambda t}).
\end{gathered}
\end{equation}
\end{theorem}
The proof will be given below in Sec.~\ref{ss21}.

We need some estimates for the symbols \eqref{fofxi}. These are
provided by the following lemma.
\begin{lemma}\label{l1.3}
The following estimates hold for the functions \eqref{fofxi}\textup:
\begin{equation}\label{esti-fj}
\begin{gathered}
    \abs{f_1^{(k)}(\xi)} \le C_{k0}(1+\abs{\xi})^{-1-k},\qquad
    \abs{f_2^{(k)}(\xi)} \le C_{km}(1+\abs{\xi})^{-3/2-k},\\
    \left\vert \pd{^{k+m}f_3(\xi,t)}{x^k\partial t^m}\right\vert
     \le C_{km}e^{-\nu t}(1+\abs{\xi})^{-1/2-k},
\end{gathered}
\end{equation}
$k=0,1,2,\dotsc$, where the $C_{km}$ are some constants \textup(in
general, different from those introduced earlier\textup).
\end{lemma}
\begin{proof}
For $k=0$, the desired estimates \eqref{esti-fj} readily follow
from \eqref{esti-g0} and \eqref{esti-g00}; it suffices to replace
$\xi$ by $\xi^{1/2}$ (and use the fact that the functions $f_j$
given by \eqref{fofxi} are smooth and in particular continuous at
$\xi=0$). Next, note that if $f(\xi)=F(\xi^{1/2})$, where
$F(\zeta)$ is a smooth even function, then
$f'(\xi)=\Psi(\xi^{1/2})$, where
$\Psi(\zeta)=\frac12F'(\zeta)/\zeta$ is again a smooth even
function. Thus, it suffices to prove that if a smooth even function
$F$ satisfies estimates of the form
\begin{equation*}
    \abs{F^{(k)}(\zeta)}\le d_k (1+\abs{\zeta})^{-k-k_0},\qquad k=0,1,2,\dots\,,
\end{equation*}
for some $k_0$, then $\Psi$ satisfies the same estimates but with
$k_0$ increased by $2$ and with new constants $d_k$, each of which
is a finite linear combination of the old ones. This is trivial for
$\abs{\zeta}\ge1$, and in the region $\abs{\zeta}<1$ one can use
the identity
\begin{equation*}
    \zeta^{-1}F'(\zeta)=\zeta^{-1}\bigl(F'(\zeta)-F'(0)\bigr)
    =\int_0^1 F''(\theta\zeta)\,d\theta.\qedhere
\end{equation*}
\end{proof}

\subsection{Computation of the transient part and the equivalent sources}\label{ss21}

Now we will prove Lemmas~\ref{l2.1} and~\ref{l2.2} by proving the
equivalent Theorem~\ref{th-asf}. Let $f(\xi)$ be any of the
functions $f_1(\xi)$, $f_2(\xi)$, and $f_3(\xi,t)$ given
by~\eqref{fofxi} or the function $f_4(xi,t)=\partial f_3(\xi,t)/\partial t$.
We need to compute the difference
\begin{equation}\label{difference}
{\mathcal{R}}\equiv{\mathcal{R}}(x,\lambda,\mu)
=\bigl(f(\lambda^{-2}L)-f(\lambda^{-2}L^{(0)})\bigr) V\left(\frac
x\mu\right)
\end{equation}
and estimate it in an appropriate norm. Let us make the change of
variables $x=\mu y$. In the new variables, \eqref{difference}
becomes
\begin{equation}\label{difference-y}
 {\mathcal{R}}=\bigl(f(L_y)-f(L_y^{(0)})\bigr)V(y),
\end{equation}
where
\begin{equation*}
L_y=-\omega^2\nabla_y\frac{c^2(\mu y)}{c_0^2}\nabla_y, \qquad
L_y^{(0)}=-\omega^2 \nabla_y^2, \qquad
\nabla_y=\left(\frac\partial{\partial y_1},\frac\partial{\partial
y_2}\right),
\end{equation*}
and $\omega=c_0/(\lambda\mu)$ is bounded by condition~\eqref{eq1.1e}.

To compute this difference, we use the machinery of noncommutative
analysis. We refer the reader to \cite{14,15} for details
concerning the definition and properties of functions of
noncommuting operators and only recall that a function
$F(\A1_1,\dots,\A n_n)$ of (possibly, noncommuting) operators
$A_1,\dots A_n$ can be defined as follows in the particular case
where the $A_j$ are the generators of uniformly bounded strongly
continuous one-parameter operator groups $e^{iA_j t}$, $t\in\mathbf{R}$,
on a Hilbert space $H$:
\begin{equation*}
    F(\A1_1,\dots,\A n_n)u=\frac{1}{(2\pi)^{n/2}}
    \int_{\mathbf{R}^n}\widetilde F(t_1,t_2\dots,t_n)e^{iA_n t_n}\dotsm  e^{iA_1
    t_1}u\,dt_1\dotsm dt_n,
\end{equation*}
$u\in H$, where $\widetilde F$ is the Fourier transform of the symbol $F$,
which is assumed to satisfy certain conditions (e.g., see
\cite{15}) guaranteeing that the integral on the right-hand side is
well defined. The numbers (Feynman indices) over operators indicate
the order of their action: of any two operators, the operator with
the smaller Feynman index stands to the right of the operators with
the larger Feynman index in products.

It follows by the zero-order Newton formula of noncommutative
analysis (see \cite{14} and \cite[Theorem I.8]{15}) that
\begin{equation}
f(L_y)-f(L_y^{(0)})=\frac{\delta
f}{\delta\xi}(\overset3{L_y},\overset1{L_y^{(0)}})\overset2{\overline{(L_y-L_y^{(0)})}}
=\frac{\delta f}{\delta\xi}(\overset3{L_y},\overset1{L_y^{(0)}})
\overset2{T}, \label{eq3.2}
\end{equation}
where
\begin{equation*}
\frac{\delta
f}{\delta\xi}(\xi_1,\xi_2)=\frac{f(\xi_1)-f(\xi_2)}{\xi_1-\xi_2}
\end{equation*}
is the first difference quotient of $f$ and
\begin{equation*}
    T=L_y-L_y^{(0)}=\omega^2\biggl\langle\nabla_y,
    \biggl(1-\frac{c^2(\mu y)}{c_0^2}\biggr)\nabla_y\biggr\rangle
    \equiv \omega^2\sum_{j=1}^2\pd{}{y_j}\phi(\mu y)\pd{}{y_j}.
\end{equation*}
Here we have denoted
\begin{equation*}
    \phi(z)=1-\frac{c^2(z)}{c_0^2};
\end{equation*}
this function is uniformly bounded together with all of its
derivatives for $z\in\mathbf{R}^2$, and $\phi(0)=0$.

We further transform the right-hand side of \eqref{eq3.2} as
follows.
\begin{proposition}\label{p33}
\begin{equation}\label{eq3.2a}
\frac{\delta f}{\delta\xi}(\overset3{L_y},\overset1{L_y^{(0)}})
 \overset2{T}
 =\frac{\delta f}{\delta\xi}(\overset3{L_y},\overset2{L_y^{(0)}})
        \overset1{T}
 + \frac{\delta^2 f}{\delta\xi^2}
   (\overset4{L_y},\overset3{L_y^{(0)}},\overset1{L_y^{(0)}})
       \overset2{\overline{[T,L_y^{(0)}]}},
\end{equation}
where
\begin{equation*}
    \frac{\delta^2 f}{\delta\xi^2}(\xi_1,\xi_2,\xi_3)
    =\frac{\frac{\delta f}{\delta\xi}(\xi_1,\xi_2)
      -\frac{\delta f}{\delta\xi}(\xi_1,\xi_3)}
    {\xi_2-\xi_3}
\end{equation*}
is the second difference quotient of $f$ and
\begin{equation*}
    [T,L_y^{(0)}]=TL_y^{(0)}-L_y^{(0)}T
\end{equation*}
is the commutator of $T$ and $L_y^{(0)}$.
\end{proposition}
\begin{proof}
The proof of \eqref{eq3.2a} mimics the derivation of the general
commutation formula \cite[Proposition I.3]{15}
\begin{equation}\label{gcf}
    \overset2Af(\overset1B)-\overset1Af(\overset2B)
    =\frac{\delta f}{\delta\xi}(\overset1B,\overset3B)
    \overset2{\overline{[A,B]}}
\end{equation}
of noncommutative analysis, with $T$ playing the role of $A$ and
$L_y^{(0)}$ playing the role of $B$. Recall this derivation (e.g.,
see~\cite[pp.~52--53]{15}). We need to compute
\begin{equation*}
    [A,f(B)]\equiv Af(B)-f(B)A=\A2f(\B1)-\A1f(\B2).
\end{equation*}
The Feynman indices can be chosen independently for either summand
on the right, and we can write
\begin{equation*}
    [A,f(B)]=\A2f(\B1)-\A2f(\B3)=\A2(f(\B1)-f(\B3))
    =\A2(\B1-\B3)\df(\B1,\B3).
\end{equation*}
(Here we have used the identity $f(x)-f(\xi)=(x-\xi)\df(x,\xi)$,
which is in fact the definition of $\delta f/\delta\xi$.) Next, we
move apart the Feynman indices over the $B$'s, thus obtaining
\begin{equation*}
    \A2(\B1-\B3)\df(\B1,\B3)\!=\!\A2(\B1-\B3)\df(\B0,\B4)\!=\!
    \AB2\df(\B0,\B4)\!=\!\AB2\df(\B1,\B3).
\end{equation*}
(In the middle, we have written $\A2(\B1-\B3)=\AB2$ using the fact
that no other operators in the expression have Feynman indices in
the interval $[1,3]$.) Thus, we arrive at the desired commutation
formula \eqref{gcf}.

The derivation of \eqref{eq3.2a} differs from this only in that
now, instead of $f(\B1)$, we have $\frac{\delta f}{\delta\xi}
(\overset3{L_y}, \overset1{L_y^{(0)}})$; i.e., there is an
additional operator argument, $\overset3{L_y}$, but this argument
does not invalidate the computation, because its Feynman number
does not lie between those of $A=T$ and $B=L_y^{(0)}$.
\end{proof}

Let us evaluate the commutator $[T,L_y^{(0)}]$.
\begin{proposition}\label{p34}
One has
\begin{equation*}
    [T,L_y^{(0)}]=\mu T_1, \quad\text{where}\quad
    \norm{T_1\colon H^s(\mathbf{R}^2_y)\longrightarrow H^{s-3}(\mathbf{R}^2_y)}\le C_s
\end{equation*}
for all $s$ with some constants $C_s$ independent of $\mu$ as
$\mu\to0$.
\end{proposition}
\begin{proof}
We have
\begin{align*}
    [T,L_y^{(0)}]&=\omega^4\langle\nabla_y,[\nabla_y^2,\phi(\mu y)]
    \\
    &=\mu\omega^4\sum_{j=1}^2\biggl\langle\nabla_y,
    \biggl(2\phi'_{z_j}(\mu y)\pd{}{y_j}
    +\mu\phi''_{z_jz_j}(\mu y)\biggr)\nabla_y\biggr\rangle,
\end{align*}
and it remains to recall that $\phi(z)$ is uniformly bounded
together with all derivatives.
\end{proof}

By Propositions \ref{p33} and \ref{p34}, we can write
\begin{equation*}
    f(L_y)-f(L_y^{(0)})=\frac{\delta f}{\delta\xi}(\overset3{L_y},\overset2{L_y^{(0)}})
        \overset1{T}
 + \mu\frac{\delta^2 f}{\delta\xi^2}
   (\overset4{L_y},\overset3{L_y^{(0)}},\overset1{L_y^{(0)}})
       \overset2T_1.
\end{equation*}
Accordingly,
\begin{equation}\label{metka1}
    {\mathcal{R}}=\bigl(f(L_y)-f(L_y^{(0)})\bigr)V=AW+\mu BV,
\end{equation}
where
\begin{equation}\label{metka1a}
    W=TV,\qquad
    A=\frac{\delta f}{\delta\xi}(\overset2{L_y},\overset1{L_y^{(0)}}),
    \qquad
    B=\frac{\delta^2 f}{\delta\xi^2}
   (\overset4{L_y},\overset3{L_y^{(0)}},\overset1{L_y^{(0)}})
       \overset2T_1.
\end{equation}
Let us estimate the expression \eqref{metka1} for $f=f_j$,
$j=1,2,3,4$.
\begin{proposition}\label{p32}
One has $V\in H^s(\mathbf{R}_y^2)$ for every $s$.
\end{proposition}
\begin{proof}
This follows from the estimates~\eqref{eq1.1c}.
\end{proof}
\begin{proposition}\label{p35}
For every $s$, one has $W\in H^s(\mathbf{R}_y^2)$ and
\begin{equation*}
    \norm{W}_{H^s(\mathbf{R}^2_y)}=O(\mu),\qquad \mu\to0.
\end{equation*}
\end{proposition}
\begin{proof}
We have
\begin{equation*}
    \phi(\mu y)=\mu\langle F(\mu y),y\rangle,
\end{equation*}
where the vector function
\begin{equation*}
    F(z)=\int_0^1\pd{\phi}{z}(\theta z)\,d\theta
\end{equation*}
is bounded together with all derivatives, and hence for the
function $W=TV$ we obtain
\begin{equation*}
    W(y)=\mu\omega^2\biggl(\pd{}{y_1}\langle F(\mu y),y\rangle\pd{V(y)}{y_1}
    +\pd{}{y_2}\langle F(\mu y),y\rangle\pd{V(y)}{y_2}\biggr).
\end{equation*}
Since, by virtue of the estimates \eqref{eq1.1c}, the function
$y_j\partial V(y)/\partial y_k$ lies in $H^s(\mathbf{R}^2_y)$ for every $s$, we
arrive at the desired assertion.
\end{proof}

\begin{proposition}\label{p36}
Let $f=f_j$, $j=1,2,3,4$. Then for each $s\in\mathbf{R}$ there exists a
constant $C_s$ independent of $\mu\to0$ such that
\begin{align*}
    \norm{A \colon H^s(\mathbf{R}^2_y)\to H^s(\mathbf{R}^2_y)}&\le
    C_s,\qquad \norm{B\colon H^s(\mathbf{R}^2_y)\to H^{s-3}(\mathbf{R}^2_y)}\le
    C_s\\ \intertext{for $j=1,2$,}
\norm{A \colon H^s(\mathbf{R}^2_y)\to H^s(\mathbf{R}^2_y)}&\le
    C_se^{-\nu t},\; \norm{B\colon H^s(\mathbf{R}^2_y)\to H^{s-3}(\mathbf{R}^2_y)}\le
    C_se^{-\nu t}
\end{align*}
for $j=3,4$.
\end{proposition}

\begin{proof}
We make use of the following representation of the $k$th difference
quotient:
\begin{multline*}
    \frac{\delta^k f}{\delta\xi^k}(\xi_1,\dots,\xi_{k+1})
    =\int_{\Delta_k}f^{(k)}(\theta_1\xi_1+\dotsm+\theta_{k+1}\xi_{k+1})\,d\theta_1\dotsm
    d\theta_k\\
    =\frac1{\sqrt{2\pi}}\int_{\Delta_k}\biggl(\int_{-\infty}^\infty
    \widetilde{f^{(k)}}(p)e^{ip(\theta_1\xi_1+\dotsm+\theta_{k+1}\xi_{k+1})}dp\biggr)\,d\theta_1\dotsm
    d\theta_k,
\end{multline*}
where $\widetilde{f^{(k)}}(p)$ is the Fourier transform of the $k$th
derivative $f^{(k)}(\xi)$ and
\begin{equation*}
  \Delta_k=\{(\theta_1,\dotsc,\theta_{k+1}\in\mathbf{R}^{k+1}
  \colon \theta_1+\dotsm+\theta_{k+1}=1,\;\theta_j\ge0,\;
   j=1,\dots,k+1\}
\end{equation*}
is the standard $k$-simplex. Hence
\begin{equation}\label{expr}
\begin{aligned}
  \frac{\delta f}{\delta\xi}(\overset2{L_y},\overset1{L_y^{(0)}})
    &=\frac1{\sqrt{2\pi}}\int_{\Delta_1}\biggl(\int_{-\infty}^\infty
    \widetilde{f'}(p)e^{ip\theta_1L_y}e^{ip\theta_2L_y^{(0)}}dp\biggr)\,d\theta_1
\\
  \frac{\delta^2 f}{\delta\xi^2}(\overset4{L_y},
    \overset3{L_y^{(0)}}\!\!,&\overset1{L_y^{(0)}})\overset2T_1
    \\&=
    \frac1{\sqrt{2\pi}}\int_{\Delta_2}\biggl(\int_{-\infty}^\infty
    \widetilde{f''}(p)e^{ip\theta_1L_y}e^{ip\theta_2L_y^{(0)}}
    T_1e^{ip\theta_3L_y^{(0)}}dp\biggr)\,d\theta_1
    d\theta_2.
\end{aligned}
\end{equation}
Let us estimate the operators~\eqref{expr}. To this end, we use the
following lemma.
\begin{lemma}\label{l6}
For each $s$, there exists a constant $\widetilde C_s$ independent of
$\mu\to0$ such that
\begin{equation*}
    \norm{e^{itL_y}\colon H^s(\mathbf{R}^2_y)\to H^s(\mathbf{R}^2_y)}\le
    \widetilde C_s,\qquad \norm{e^{itL_y^{(0)}}\colon H^s(\mathbf{R}^2_y)\to H^s(\mathbf{R}^2_y)}\le
    \widetilde C_s
\end{equation*}
for all $t\in\mathbf{R}$.
\end{lemma}
\begin{proof}
For $s=0$, the claim is obvious, because the operators $L_y^{(0)}$
and $L_y$ are self-adjoint in $L^2(\mathbf{R}_y)$. For other values of
$s$, one equips $H^s(\mathbf{R}^2_y)$ with the equivalent norm
$\norm{(1+L_y)^{s/2}u}$, so that the operator $L_y$ becomes
self-adjoint. This norm depends on the parameter $\mu$, but it is
not hard to prove (for positive integer $s$ by a straightforward
computation, and for other $s$ by duality and interpolation) that
the constants in the inequalities specifying the equivalence of
norms remain bounded as $\mu\to0$. The argument for $L_y^{(0)}$ is
simpler, because the parameter $\mu$ is not involved. The proof of
Lemma~\ref{l6} is complete.
\end{proof}

Now we can finish the proof of Proposition~\ref{p36}. If
$f=f_1,f_2,f_3,$ or $f_4$, then it follows from Lemma~\ref{l1.3}
that the Fourier transforms of $f'$ and $f''$ belong to $L^1(\mathbf{R})$,
and in the case of $f_3$ and  $f_4$ the $L^1$-norm decays as
$e^{-\nu t}$. By combining this with Lemma~\ref{l6} and with the
estimate for $T_1$ in Proposition~\ref{p34}, we arrive at the
assertion of Proposition~\ref{p36}.
\end{proof}

By applying Propositions \ref{p32}, \ref{p35}, and \ref{p36} to
formulas~\eqref{metka1} and~\eqref{metka1a}, we find that
${\mathcal{R}}=O(\mu)$ in all $H^s(\mathbf{R}^2_y)$ for $f=f_1$ and $f=f_2$ and
${\mathcal{R}}=O(\mu e^{-\nu t})$ in all $H^s(\mathbf{R}^2_y)$ for $f=f_3$ and
$f=f_4$. Let us finally estimate the remainders $R_j$
in~\eqref{eqsotraf-a}. We should take into account the additional
factor $\lambda^{-1}$ for $j=2$ and pass from the variables $y$ to the
original variables $x=\mu y$. Since
\begin{equation}\label{ququ}
    \norm{u}_s\equiv\norm{u}_{H^s(\mathbf{R}^2_x)}\le \mu^{1-s}\norm{u}_{H^s(\mathbf{R}^2_y)}
    \qquad\text{for $\mu\le1$ and $s>0$,}
\end{equation}
we arrive at the desired estimates \eqref{esti-a}. For example, for
$R_2$ we obtain
\begin{equation*}
    \norm{R_2}_2 \le C\mu \lambda^{-1}\mu^{1-2}=C\lambda^{-1}\le
    \frac{C\omega}{c_0}\mu
\end{equation*}
(where the factor $\lambda^{-1}$ comes from \eqref{eqsotraf} and the
factor $\mu^{1-2}=\mu^{-1}$ from \eqref{ququ} for $s=2$). The
estimates for $R_1$ and $R_3$ are similar. The proof of
Theorem~\ref{th-asf} and hence of Lemmas~\ref{l2.1} and~\ref{l2.2}
is complete. \qed

\section{Examples}\label{s5}

In conclusion, let us present two simple examples in which the
asymptotics of the solution of the Cauchy problem \eqref{eq1.1},
\eqref{eq1.2} with a special right-hand side will be demonstrated.
Namely, we use the right-hand side \eqref{eq1.1b},
$Q(x,t)=\lambda^2g_0'(\lambda t)V(x/\mu)$, where
$V(y)=A(1+(y_1/b_1)^2+(y_2/b_2)^2)^{-3/2}$ is the simplest spatial
shape factor~\eqref{eq1.1f} and the function $g_0(\tau)$ is given
by one of formulas (a) (a \textit{sine source}) and (b) (a
\textit{polynomial source}) in Eq.~\eqref{source}.

Recall that the asymptotics of the solution is given by
Theorem~\ref{thm1}, Eqs.~\eqref{sol_tran}
and~\eqref{sol_tran_polar} (the transient solution component) and
by Theorem~\ref{p4.1}, Eq.~\eqref{4.2} (the propagating solution
component away from the focal points). The transient component
$\eta_{trans}(x,t)$ and the wave profile $F(z,\psi)$
(see~\eqref{4.3}) of the propagating component depend only on the
right-hand side and on the parameters $\lambda,\mu$, and $\omega$; they
are represented by integrals which, for our choice of the
right-hand side, can be evaluated (or, in the case of the transient
component, considerably simplified) analytically. The other
ingredients of the asymptotic formula~\eqref{4.2} for the
propagating component (the phase functions $S_j(x,t)$, the
Lagrangian coordinates $\psi_j(x,t)$, the Morse index
$m(\psi_j^0,t)$, and the factors responsible for the Green law and
for the trajectory divergence) depend on the solution of the Cauchy
problem \eqref{4.1} for the Hamiltonian system \eqref{eq3.5},
which, except for the simplest cases, should be solved numerically.

Accordingly, our exposition in both examples is as follows. First,
we find the function $G_0(\xi,t)$ \eqref{G0}, which plays a crucial
role in all subsequent calculations. Then we write out the wave
profile $F(z,\psi)$ and finally present the expression for the
transient component $\eta_{trans}(x,t)$ of the solution. In the
second example, we also numerically compute the trajectories and
display snapshots of the solution obtained with the use of
\textit{Wolfram Mathematica}.

The calculations are mostly carried out in polar coordinates, so
let us rewrite formula~\eqref{FtV} for the Fourier transform of $V$
in the polar coordinates $(\rho,\psi)$, where $p=\rho \mathbf
n(\psi)$ with $\mathbf n(\psi) = (\cos\psi, \sin\psi)$:
\begin{equation}\label{F_V}
\widetilde V\big(\rho\mathbf n(\psi)\big) = Ab_1b_2e^{-\rho
\beta(\psi)}, \quad\text{where } \beta(\psi) \equiv
\sqrt{b_1^2\cos^2\psi+b_2^2\sin^2\psi}.
\end{equation}

\subsection{The case of a sine source}

Let
\begin{equation*}
  g_0(\tau) = ae^{-\tau}(\sin(\alpha\tau+\phi_0)-\sin\phi_0)
\end{equation*}
where $a = (\alpha^2+1) / (\alpha\cos\phi_0-\alpha^2\sin\phi_0)$ is
a normalizing factor. By evaluating the integral in~\eqref{G0}, we
obtain
\begin{equation}\label{2.7}
G_0(\xi,t) = ae^{-t} \Bigl( \frac{ i  e^{-i(\alpha t + \phi_0) }/2}
{1 + i \alpha + i\xi}-  \frac{ i e^{i(\alpha t + \phi_0)}/2} {1 -
i\alpha + i\xi} - \frac{\sin \phi_0} {1 + i {\xi}} \Bigr).
\end{equation}

We see that $G_0(\xi,t)$ is a rational function of $\xi$. Moreover,
a routine computation (which we omit) shows that it can be
represented in the form
\begin{equation}\label{2.7a}
    G_0(\xi,t)=\sum_m q_m(t)\bigl(R_m(\xi^2)+i\xi Q_m(\xi^2)\bigr),
\end{equation}
where $R_m(\zeta)$ and $Q_m(\zeta)$ are rational functions with
real coefficients and with denominators nonvanishing for
$\zeta\ge0$. This is, of course, consistent with the assertion in
Lemma~\ref{l44} concerning the parity of the real and imaginary
parts of $G_0$. As to $\widetilde g_0(\xi)$, we have
\begin{equation}\label{2.7b}
\widetilde g_0(\xi)=\frac1{\sqrt{2\pi}}G_0(\xi,0) =
\frac{a}{\sqrt{2\pi}}\Big( \frac{ i  e^{-i\phi_0}/2} {1 + i \alpha
+ i\xi}-  \frac{ i e^{i\phi_0}/2} {1 - i\alpha + i\xi} - \frac{\sin
\phi_0} {1 + i {\xi}} \Big).
\end{equation}

To evaluate the wave profile $F(z,\psi)$ of the propagating
solution component, we substitute the functions \eqref{F_V} and
\eqref{2.7b} into formula~\eqref{4.3} and obtain
\begin{multline*}
F(z,\psi)= \frac{aAb_1b_2
e^{-i\pi/4}}{\sqrt{2\pi}\omega^{3/2}}\\\shoveright{\times
\int_0^\infty \sqrt{\rho} \Big( \frac{i
 e^{-i\phi_0}/2} {1 + i\alpha - i\rho} - \frac{i
e^{i\phi_0}/2} {1 - i \alpha - i\rho} - \frac{\sin \phi_0} {1 - i
{\rho}} \Big) e^{-\rho \omega^{-1}(\beta(\psi)-iz)} d\rho}\\
\shoveleft{\phantom{F(z,\psi)}= \frac{aAb_1b_2 e^{-i
\pi/4}}{\sqrt{2\pi}\omega^{3/2}} \Big[\frac i 2 e^{-i\phi_0}
\mathrm{I}_0\big(\omega^{-1}(\beta(\psi)-iz),1+i\alpha\big)}
\\ {}-
\frac i 2  e^{i\phi_0} \mathrm{I}_0\big(\omega^{-1}(\beta(\psi)-iz),
1-i\alpha\big) - \mathrm {I}_0\big(\omega^{-1}(\beta(\psi)-iz),
1\big)\sin \phi_0 \Big],
\end{multline*}
where the integral
\begin{equation}\label{I0}
    \mathrm {I}_0(C_1,C_2) =\int_0^\infty \frac { \sqrt{\rho}e^{-C_1 \rho}\,
    d\rho} {C_2 - i\rho} ,
    \quad
    C_1, C_2 \in \mathbf{C},\; \re C_1 > 0,\;\arg C_2 \neq
    \frac\pi2,
\end{equation}
can be expressed via the complementary error function
\begin{equation*}
 \operatorname{erfc}(w)= \frac 2 {\sqrt \pi} \int_w^\infty e^{-v^2}\,dv
\end{equation*}
by the formula
\begin{equation*}
 \mathrm{I}_0(C_1,C_2)= \frac {i\sqrt \pi}{\sqrt C_1} + e^{-i\pi/4}\pi\sqrt{C_2}e^{iC_1C_2}
\operatorname{erfc}\big(e^{i\pi/4}\sqrt{C_1 C_2}\big).
\end{equation*}

To evaluate the transient term of the solution, we substitute the
functions \eqref{F_V} and \eqref{2.7} into~\eqref{sol_tran_polar}
and obtain
\begin{align*}
& \eta_{trans}(r \mathbf{n}(\varphi))
\\
& = \frac {aAb_1b_2  e^{-\lambda t}} {2\pi \omega^2} \int_0^{2\pi}
{\rm Re}\Bigl[ \frac i z (\sin (\alpha \lambda t +\phi_0) - \sin
\phi_0) + \sin \phi_0 \int_0^\infty \frac {e^{-\rho
z}d\rho}{\rho-i}
\\
& \qquad + \frac{\alpha - i}2  e^{-i(\alpha t + \phi_0)}
\int_0^\infty  \frac {e^{-\rho z}d\rho}{\rho+\alpha-i}
\\
& \qquad  + \frac{\alpha + i}2 e^{i(\alpha t + \phi_0)}
\int_0^\infty \frac {e^{-\rho z}d\rho
}{\rho-\alpha-i}\Bigr] d\psi + O(\mu)
\\
& = \frac {aAb_1b_2  e^{-\lambda t}} {2\pi \omega^2} \int_0^{2\pi}
{\rm Re}\Bigl[ \frac i z (\sin (\alpha \lambda t +\phi_0) - \sin
\phi_0)\\
& \qquad  + \sin \phi_0 \frac i 2 e^{-iz} \bigl(\pi + 2i{\rm Ci}
(z) - 2 {\rm Si} (z)\bigr)
\\
& \qquad + \frac{\alpha - i}2  e^{-i(\alpha t + \phi_0)} e^{(
\alpha - i)z} E_1\bigl(( \alpha - i)z\bigr)\\ & \qquad +
\frac{\alpha + i}2 e^{i(\alpha t + \phi_0)} e^{-(\alpha+i)z}
E_1\bigl(-(\alpha + i)z\bigr)\Bigr] d\psi + O(\mu),
\end{align*}
where $z=z(r,\varphi,\psi) = \omega^{-1} \bigl(\beta(\psi) - i r
\mu^{-1} \cos(\psi - \varphi)\bigr)$, ${\rm Re} (z) > 0$, and
\begin{equation*}
    E_1\bigl(z\bigr) \equiv \int_z^{+\infty} \frac
    {e^{-t}}{t}dt,\qquad
    {\rm Ci}\, (z) \equiv - \int_z^\infty \frac {\cos t }{t} dt,
\quad {\rm Si}\, (z) \equiv \int_0^z \frac {\sin t }{t} dt.
\end{equation*}

\subsection{The case of a polynomial source}

Now let
\begin{equation*}
  g_0(\tau) = e^{-\tau} P(\tau),
\end{equation*}
where
\begin{equation*}
  P(\tau) =
 \sum_{k=1}^n\tfrac{P_k}{k!}\tau^k
\end{equation*}
is a polynomial of degree $n$ with coefficients $P_k$ such that
$P_0=0$ and $\sum_{k=1}^n P_k=1$. Let us use formula \eqref{G0} for
$G_0(\xi,\tau)$. Since
\begin{equation*}
\int_0^\infty e^{-t-\tau-i\xi\tau}(t+\tau)^k\,d \tau  =
e^{-t}\biggl(t+i\frac{\partial}{\partial\xi}\biggr)^k \frac 1 {1+i\xi},
\end{equation*}
it follows that
\begin{equation}\label{2.8g}
G_0(\xi,t) = e^{- t} P \biggl(t+i\frac{\partial}{\partial\xi}\biggr)
\frac 1 {1+i\xi} ,\quad
 \widetilde g_0(\xi)=\frac1{\sqrt{2\pi}}P
 \biggl(i\frac{\partial}{\partial\xi}\biggr)\frac 1 {1+i\xi},
\end{equation}
and we see that $G_0(\xi,t)$ again has the form~\eqref{2.7a}. Using
\eqref{F_V}, \eqref{2.8g} and~\eqref{4.3}, we evaluate the wave
profile of the propagating part of the solution as follows:
\begin{multline*}
F(z,\psi) = \frac {A b_1 b_2 e^{-i\pi/4}} {\sqrt{2\pi}\omega^{3/2}}
\biggl[P\biggl(-\pd{}{C_2}\biggr) \mathrm {I}_0(\rho(\beta(\psi) - i
z)/\omega,C_2)\biggr]\bigg|_{C_2=1}
\\ =
-i \frac {A b_1 b_2 \sqrt \pi }{{\sqrt{2} \omega^{3/2}}}e^{i C_1  }
\biggl[P\Bigl(
 -\overset2C_1
\Bigl(\overset1{\overline{i + \frac 1 {2C_1} + \frac{d}{d
C_1}}}\Bigr)\Bigr) {\rm erfc} (\sqrt {i C_1 })
\Bigr)\biggr]\bigg|_{C_1 =\frac{\beta(\psi) - iz}\omega},
\end{multline*}
where $\mathbf{I}_0(C_1,C_2)$ is the integral \eqref{I0}.

\begin{remark}
In both examples, one can prove that the following asymptotic
formulas hold for the functions $F(z,\psi)$ for small $\omega$:
\begin{equation*}
 F(z,\psi) =
\frac{i b_1 b_2  }{2\sqrt{2}(z+i\beta(\psi))^{3/2}}+O(\omega).
\end{equation*}
This means that for small $\omega$ the solution of the
inhomogeneous problem (corresponding to ``sources stretched in
time'') passes into the solution of the homogeneous problem
(corresponding to ``instantaneous sources'').
\end{remark}

Let us compute the transient term of the solution for the case in
which $P(\tau)$ is a second-order polynomial; then
\begin{multline*}
 G_0(\xi,t) =
 e^{- t} \biggl( \frac{P_2
t^2/2 + (P_1 - P_2)t- P_1 } {1+\xi^2} + \frac{2P_2 t + 2  P_1 - 3
P_2} {(1+\xi^2)^2} + \frac{4 P_2 }{(1+\xi^2)^3}\biggr)
\\
- i\xi e^{- t} \biggl( \frac {P_2 t^2/2 + P_1 t} {1 + \xi^2} +
\frac{2P_2 t +2 P_1 - P_2} {(1+\xi^2)^2} + \frac {4 P_2}
{(1+\xi^2)^3} \biggr).
\end{multline*}
For the transient term, we find
\begin{multline*}
\eta_{trans} = - \lambda^2 e^{-\lambda t}
\Bigl[\Bigl(P_2 \lambda^2 t^2 / 2 + (P_1 - P_2) \lambda t - P_1\Bigr)
\Theta_1\Bigl(\frac{x}{\mu}\Bigr)
\\ {}+
\Bigl(2 P_2 \lambda^3 t + (2 P_1 -3 P_2) \lambda^2\Bigr)
\Theta_2\Bigl(\frac{x}{\mu}\Bigr) + 4 P_2 \lambda^4
\Theta_3\Bigl(\frac{x}{\mu}\Bigr) \Bigr],
\end{multline*}
where
\begin{align*}
\Theta_k(y,\mu) &=\frac{A b_1b_2}{2\pi\lambda^{2k}}
\int_{\mathbb{R}^2}\frac{e^{i\langle p,y\rangle}
e^{-\sqrt{(b_1p_1)^2+(b_2p_2)^2}}}{(1+(\omega|p|)^2)^k} \,d p_1d p_2.
\end{align*}
If we pass to the polar coordinates by setting $y = r \mathbf
n(\varphi)$ and $p =  \rho \mathbf n(\psi)$, then we obtain
\begin{equation*}
\Theta_k(r \mathbf n (\varphi),\mu) =
\frac{Ab_1b_2}{2\pi\lambda^{2k}} \int_{0}^\infty\int_0^{2\pi}
\frac{\rho e^{ -\rho (\beta(\psi) - i r \cos (\psi-\varphi)) }} {(1+
\omega^2\rho^2)^k} \,d\rho d \psi.
\end{equation*}
Here one can evaluate the integral over $\rho$. For $k=1,2,3$, we
obtain
\begin{multline*}
\Theta_1(r\mathbf{n}(\varphi),\mu)
= \frac{Ab_1b_2}{2\pi \lambda^2 \omega^2} \int_0^{2\pi} d\psi
\Bigl( -\cos (z)  {\rm Ci} \, (z) + \frac 1 2 \sin (z)  \bigl(\pi -
2 {\rm Si} \, (z)\bigr) \Bigr),
\\
\shoveleft{\Theta_2(r\mathbf{n}(\varphi),\mu)
 = \frac{Ab_1b_2}{8\pi \lambda^2 \omega^2} \int_0^{2\pi} d\psi
\Bigl( 2 - 2 z \sin(z) {\rm Ci}(z) - z \cos(z) \bigl(\pi - 2{\rm
Si}(z)\bigr) \Bigr),}
\\
\shoveleft{\Theta_3(r\mathbf{n}(\varphi),\mu)
 = \frac{Ab_1b_2}{32\pi \lambda^2 \omega^2} \int_0^{2\pi} d\psi
\Bigl( 4 - z \sin(z) \bigl(\pi z + 2 {\rm Ci} (z) - 2 z {\rm
Si}(z)\bigr)}
\\ {}+
z \cos(z)(-\pi + 2 z {\rm Ci}(z) + 2 {\rm Si}(z))
\Bigr),
\end{multline*}
where $z(\psi) = \omega^{-1}(\beta - ir\cos(\psi-\varphi))$.

An illustration of the solution given by the sum of propagating and
transient terms in the second example is shown in Fig.
\ref{Fig_eta_P1}. Here the propagating part is calculated for the
constant velocity $c(x) \equiv c_0 = 1$, and other constants are
$b_1 = 1, b_2 = 2, \Lambda = 1, \mu = 0.1, P_1 = 0, P_2 = 1$. The
first four snapshots are taken at small times $t = 0.3, 0.7, 1.0,
1.5$ to show how the transient term behaves, and the last three
snapshots are taken at large times $t = 1.5, 4.0, 6.5$. At $t=6.5$,
the transient term practically disappears, while the propagating
part continues its motion.The function $g_0$ and the wave profile
for $P_1=-2, P_2=3$, and various $\lambda$ are compared in
Fig.~\ref{druids}. For small $\lambda$, the wave profile has the
form that ``reproduces'' the shape of the function $g_0$, while for
large $\lambda$ the wave profile is almost the same as for $g_0 =
\delta(t)$.

\begin{figure}
\centering
\includegraphics[scale=0.45]{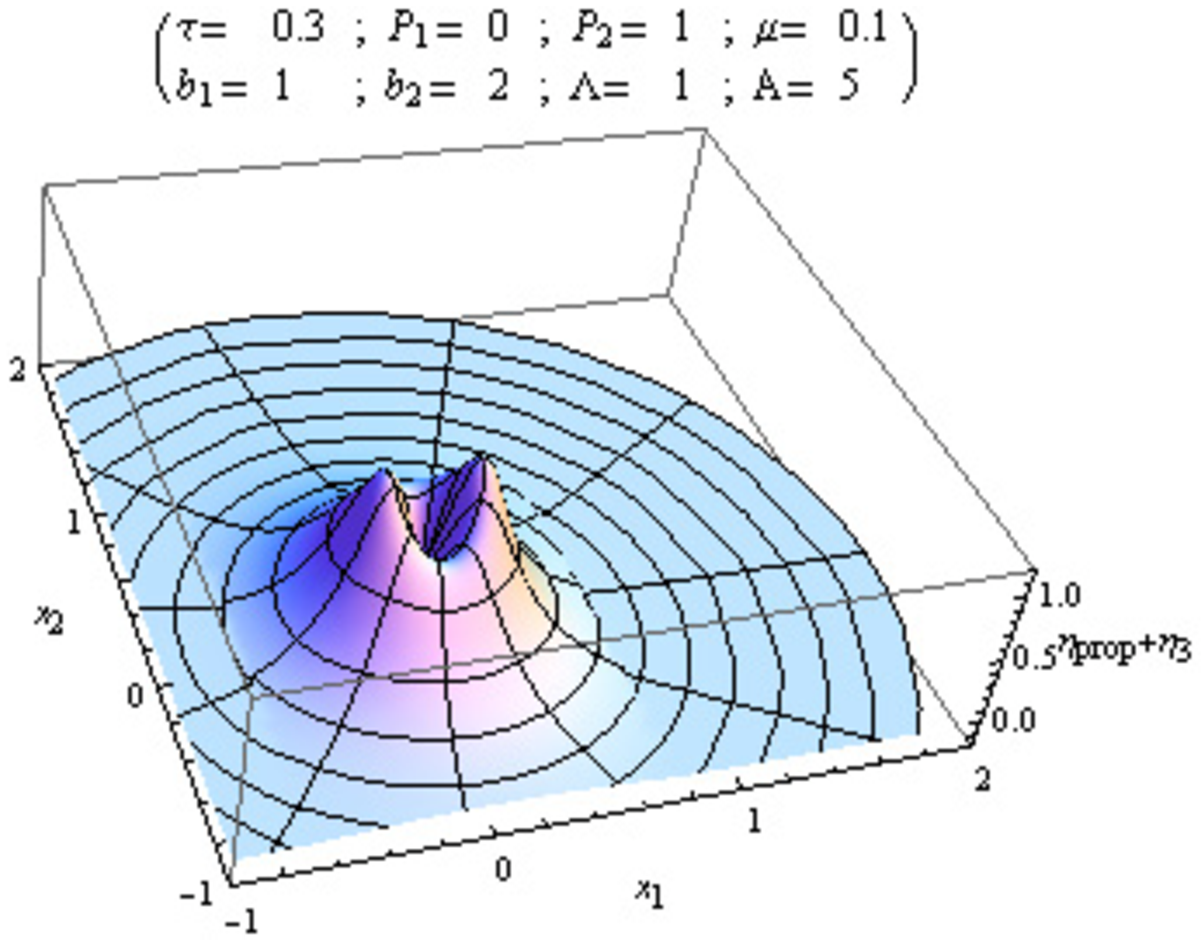}
\includegraphics[scale=0.45]{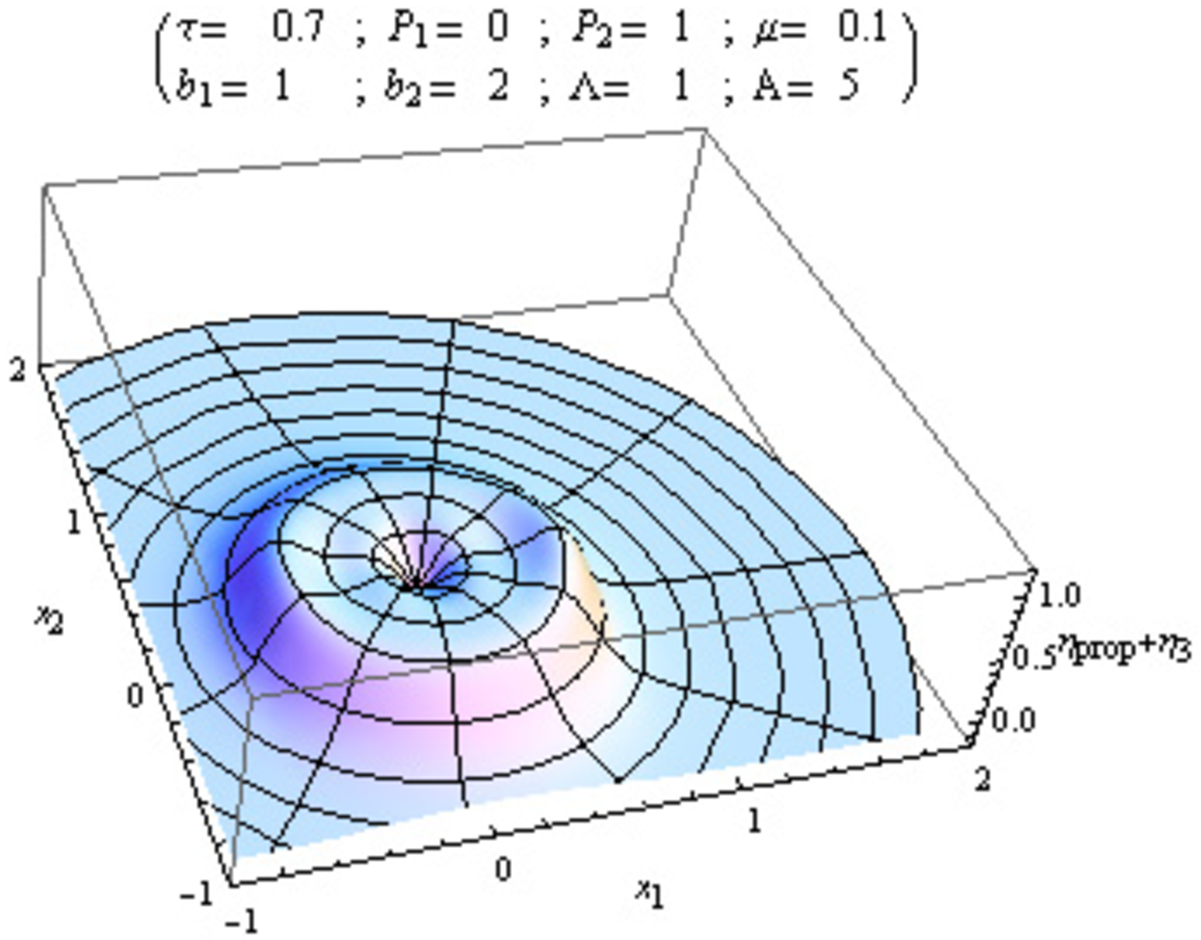}

\includegraphics[scale=0.45]{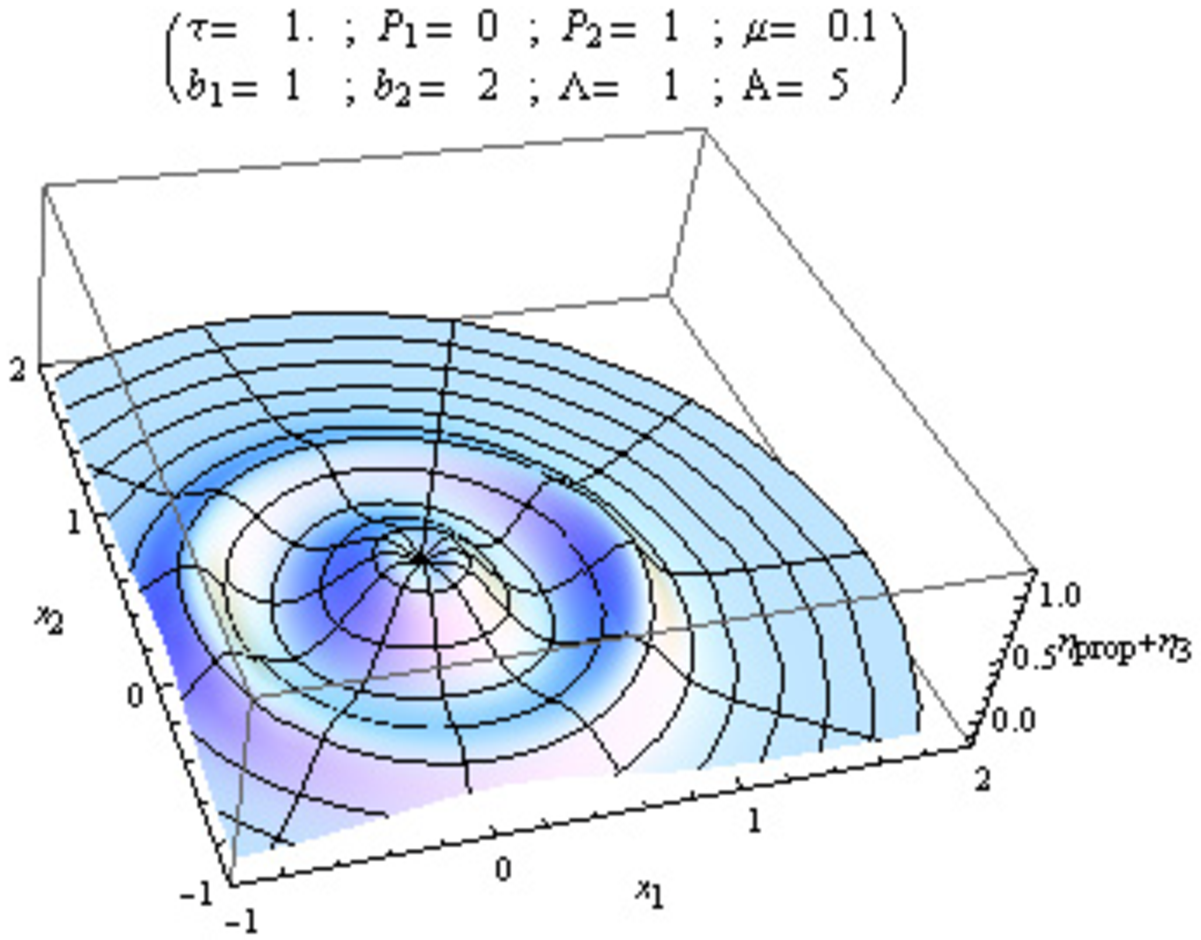}
\includegraphics[scale=0.45]{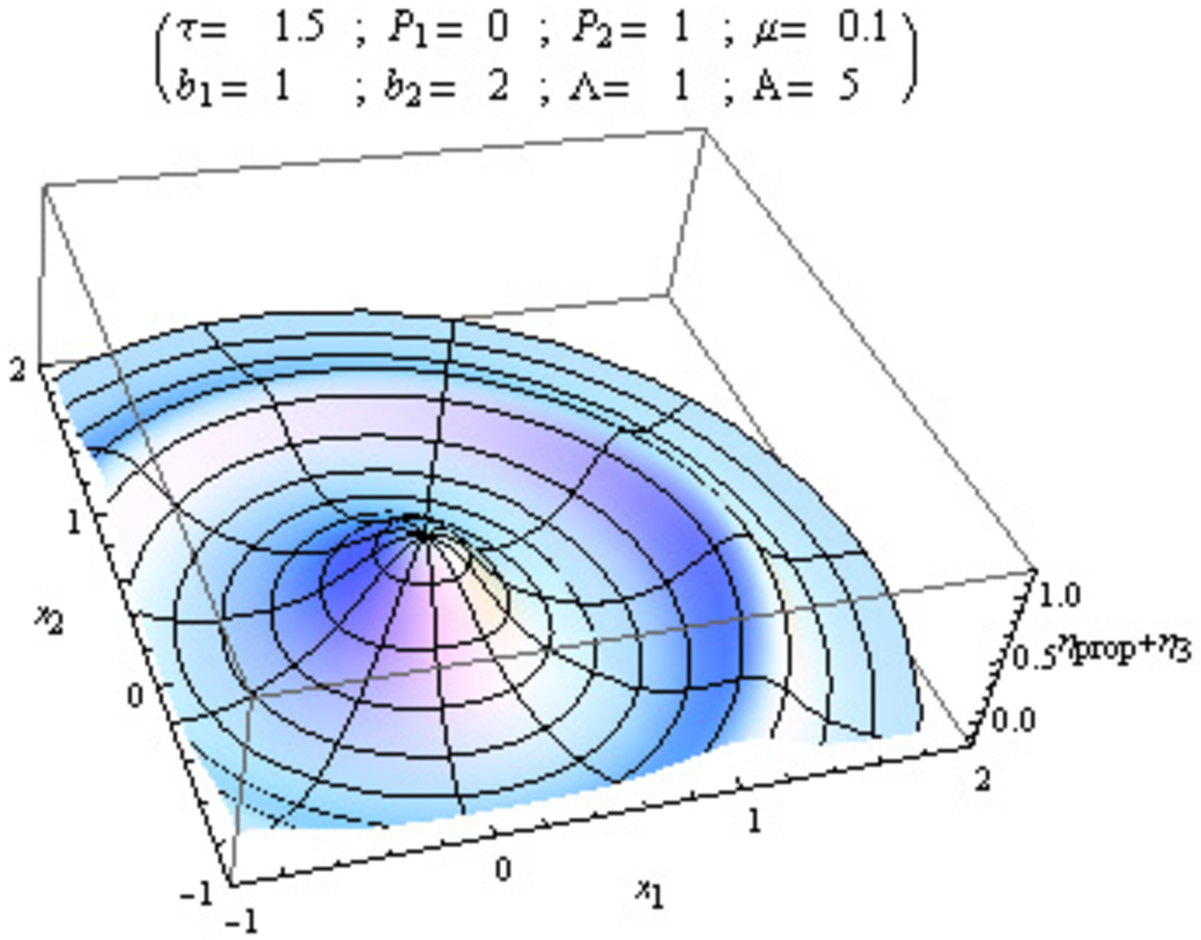}

\includegraphics[scale=0.45]{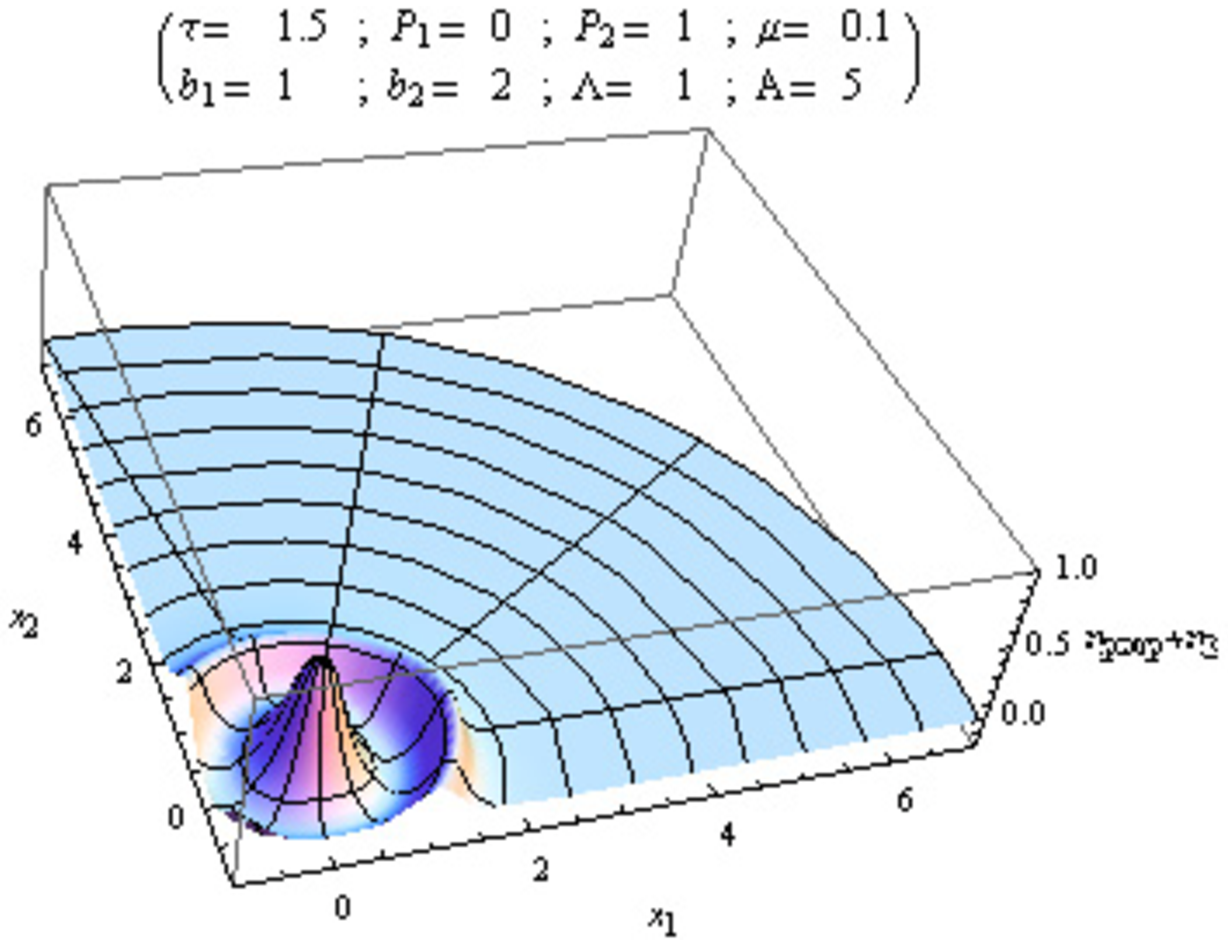}
\includegraphics[scale=0.45]{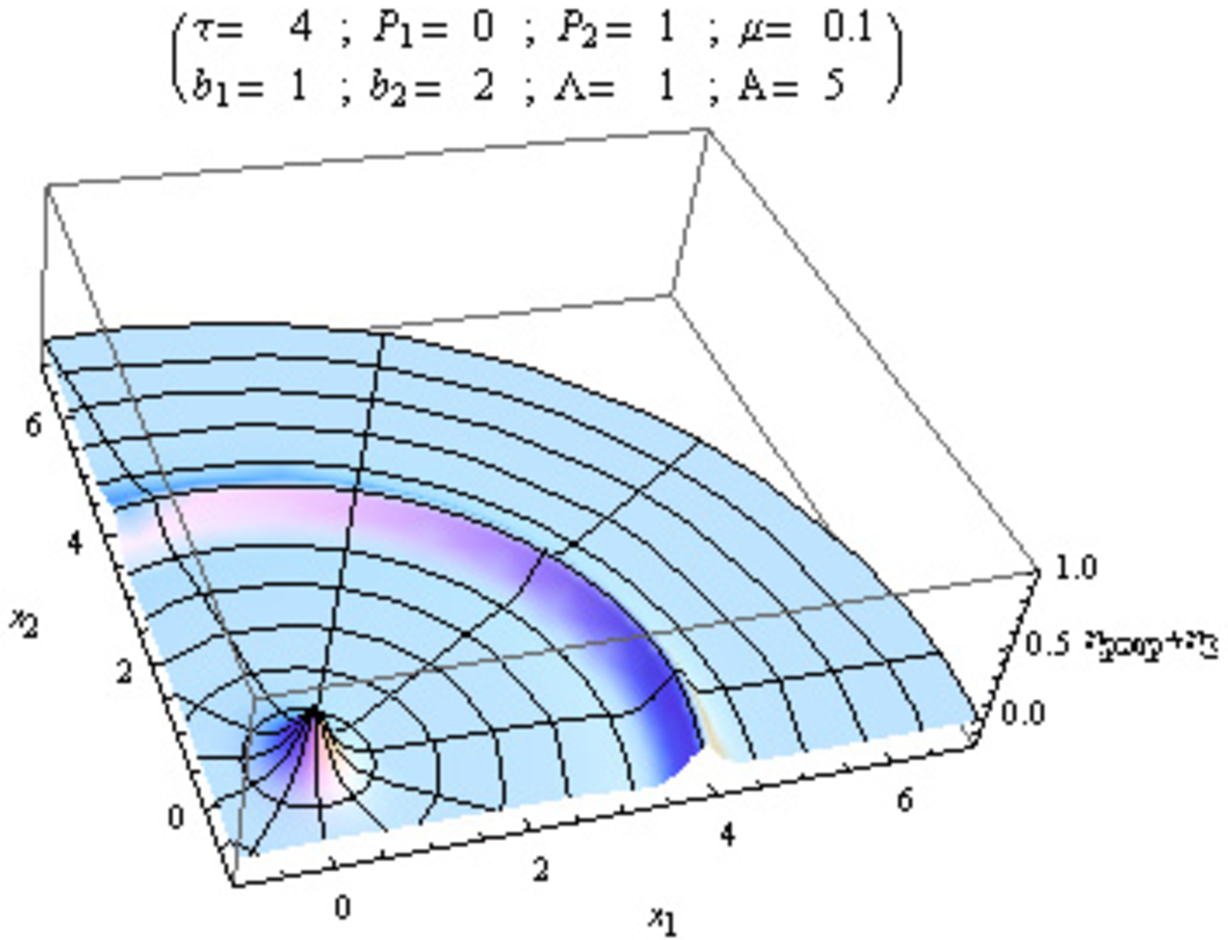}

\includegraphics[scale=0.45]{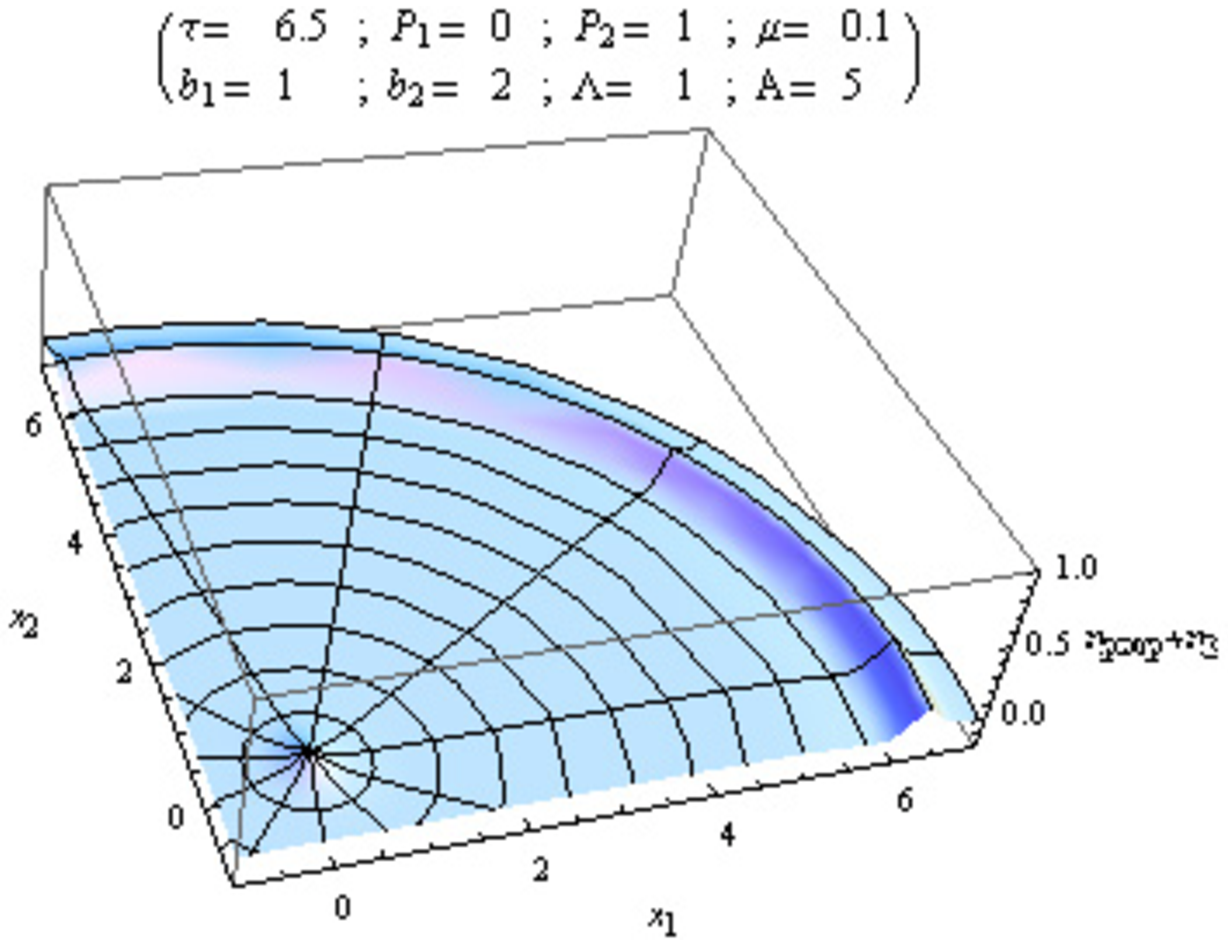}
\caption{Sum of waves $\eta_{\rm prop}+\eta_{\rm trans}$.\label{Fig_eta_P1}}
\end{figure}
\begin{figure}
\includegraphics[width=0.8\textwidth]{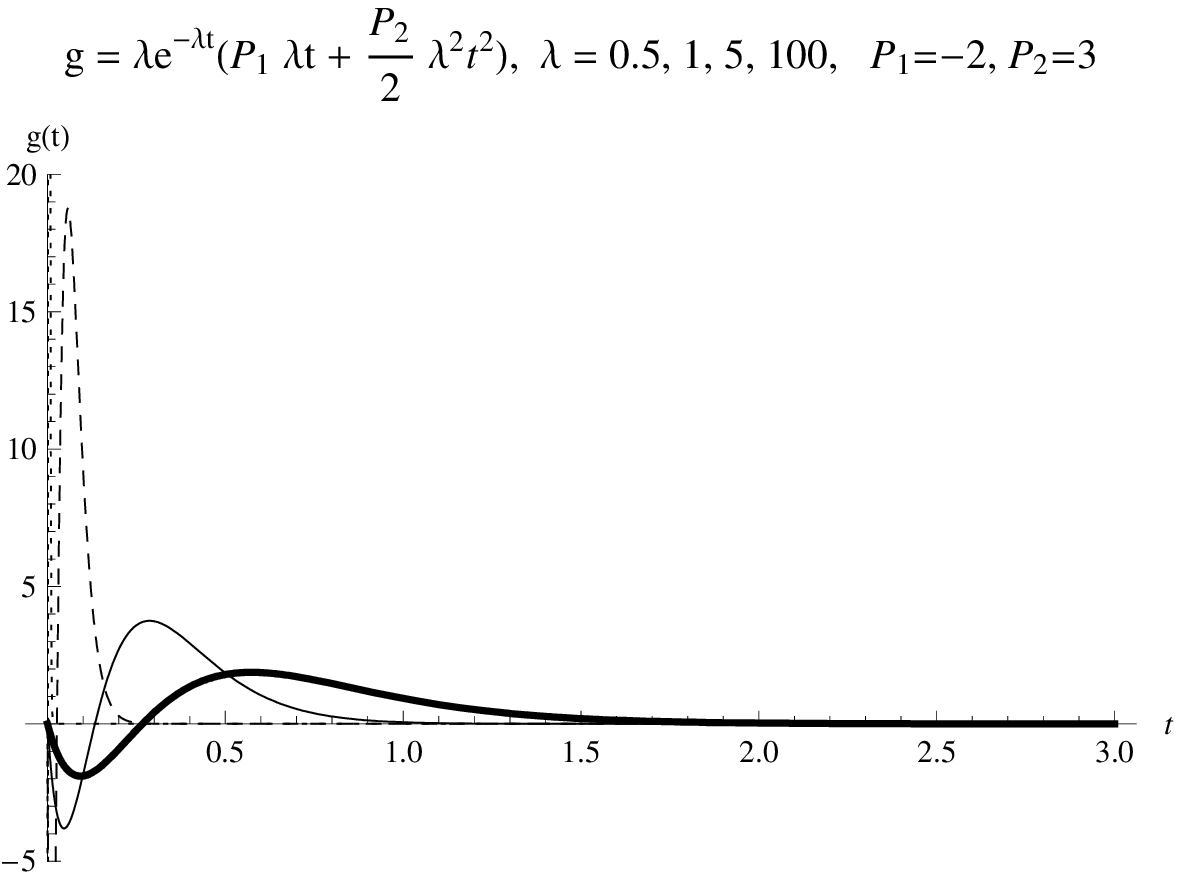}
\includegraphics[width=0.8\textwidth]{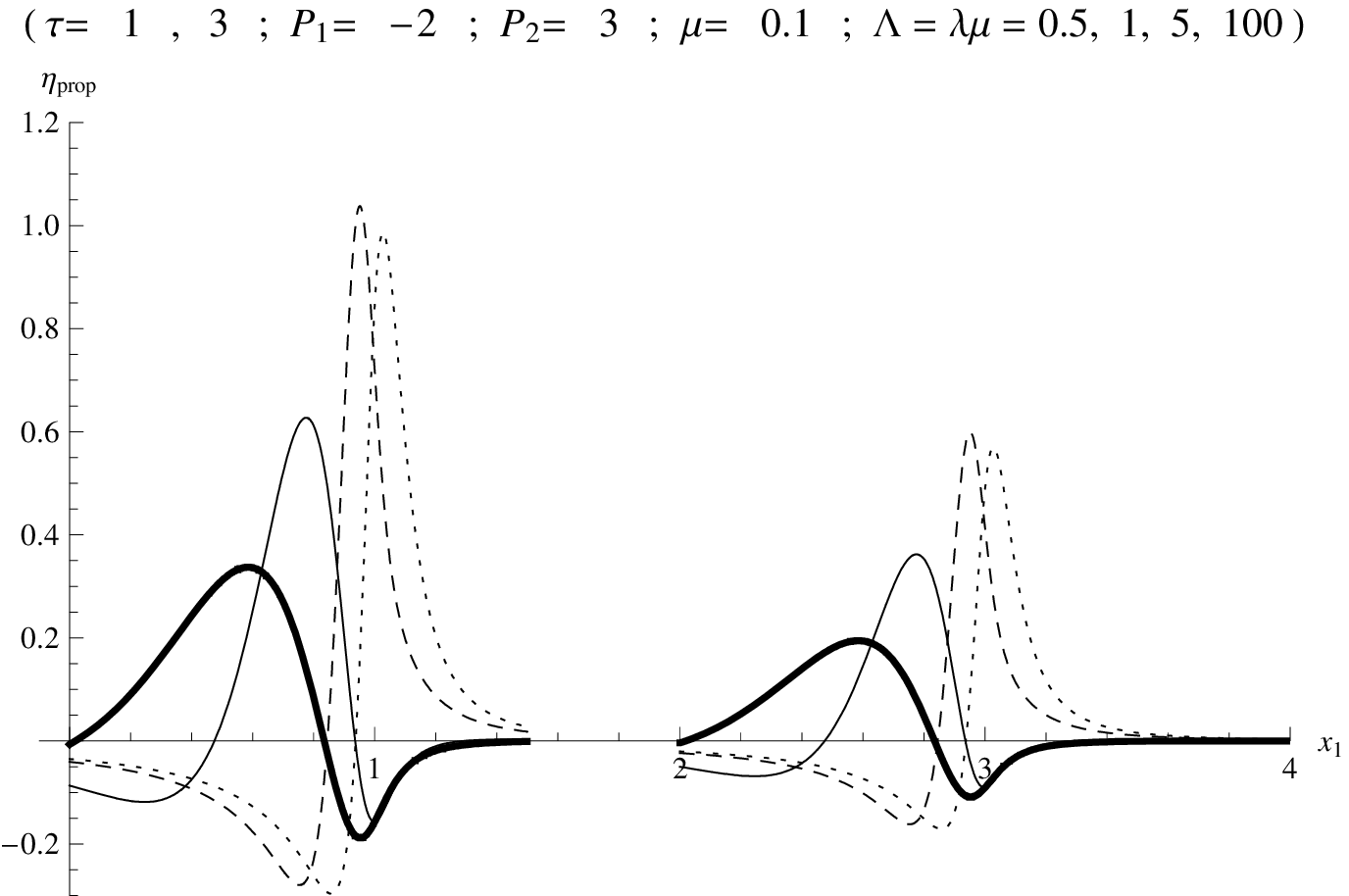}
\caption{Examples of profiles of propagating waves.\label{druids}}
\end{figure}

\end{document}